\documentclass{article}

\usepackage[english]{babel}

\usepackage[letterpaper,top=2cm,bottom=2cm,left=3cm,right=3cm,marginparwidth=1.75cm]{geometry}

\usepackage{amsmath}
\usepackage{amssymb}
\usepackage{amsthm}
\usepackage{graphicx}
\usepackage[colorlinks=true, allcolors=blue]{hyperref}
\usepackage{comment}
\usepackage{thmtools}
\usepackage{thm-restate}

\newcommand{\PC}{\mathrm{PC}_{n,d}}
\newcommand{\PCtwo}{\mathrm{PC}_{n,2}}

\newcommand{\PH}{\mathrm{PH}}
\newcommand{\PHc}{\mathrm{PH}^{\check{\mathrm{C}}}}
\newcommand{\PHvr}{\mathrm{PH}^{\mathrm{VR}}}
\newcommand{\VR}{\mathrm{VR}}
\newcommand{\C}{\check{\mathrm{C}}}

\newcommand{\preord}{\mathrm{TPO}}
\newcommand{\bd}{\mathrm{bd}}
\newcommand{\crit}{\mathrm{Crit}}
\newcommand{\critS}{\mathrm{CritS}}

\newcommand{\aff}{\mathrm{Aff}}

\renewcommand{\epsilon}{\varepsilon}
\newcommand{\QQ}{\mathbb{Q}}

\newtheorem{theorem}{Theorem}
\newtheorem{definition}[theorem]{Definition}
\newtheorem{lemma}[theorem]{Lemma}
\newtheorem{proposition}[theorem]{Proposition}
\newtheorem{corollary}[theorem]{Corollary}
\newtheorem{conjecture}[theorem]{Conjecture}

\theoremstyle{definition}
\newtheorem{example}[theorem]{Example}
\newtheorem{remark}{Remark}
\usepackage{amssymb}
\usepackage{amsthm}

\newcommand{\RR}{\mathbb{R}}
\newcommand{\NN}{\mathbb{N}}

\title{Fibers of point cloud persistence}
\author{David Beers \thanks{Department of Mathematics, University of California, Los Angeles}, Heather A Harrington \thanks{Mathematical Institute, University of Oxford; 
Center of Mathematics, Max Planck Institute of Molecular Cell Biology and Genetics; Center for Systems Biology Dresden; Faculty of Mathematics, Technische Universit\"at Dresden}, Jacob Leygonie \thanks{Mathematical Institute, University of Oxford}, Uzu Lim \thanks{School of Mathematical Sciences, Queen Mary University of London},  and Louis Theran\thanks{School of Mathematics and Statistics, University of St Andrews}}

\begin{document}
\maketitle
\begin{abstract}
Persistent homology (PH) studies the topology of data across multiple scales by building nested collections of topological spaces called filtrations, computing homology and returning an algebraic object that can be vizualised as a barcode--a multiset of intervals. The barcode is stable and interpretable, leading to applications within mathematics and data science. 
We study the spaces of point clouds with the same barcode by connecting persistence with real algebraic geometry and rigidity theory. Utilizing a semi-algebraic setup of point cloud persistence, 
we give lower and upper bounds on its dimension and provide combinatorial conditions in terms of the local and global rigidity properties of graphs associated with point clouds and filtrations. We prove that for generic point clouds in \(\mathbb{R}^d\) (\(d \geq 2\)), a point cloud is identifiable up to isometry from its VR persistence if the associated graph is globally rigid, and locally identifiable up to isometry from its \v{C}ech persistence if the associated hypergraph is rigid. 

\end{abstract}

\section{Introduction}
Persistent homology (PH) provides a multiscale geometric descriptor of data that is functorial, stable to perturbations and interpretable  \cite{cohen2005stability,chazal2009proximity,bauer2014induced}, leading to many applications in mathematics \cite{buhovsky2022coarse,polterovich2016autonomous,kislev2022bounds} and the real-world \cite{gardner2022toroidal,lee2017quantifying,miller2015fruit,benjamin2024multiscale,rabadan2019topological}. Although this process yields an interesting nontrivial descriptor of the ``shape of data'', it is unclear how much information is lost in the mapping that takes point clouds to their barcode assigned by PH. To answer this, we ask the fundamental question: What is the shape of the \textit{fibers} (i.e. level sets) of the persistent homology map?

The most common set up in topological data analysis (TDA), which we consider here, is point cloud data. Let \( P = (p_1, \ldots, p_n) \) be a configuration of ordered distinct points in \(\mathbb{R}^d\). 
While $P$ itself has no interesting topology, the union of closed balls of radius $r$ centered at each point does. The nerve theorem implies that this union has the same homotopy type as its nerve, known as the \emph{\v{C}ech} complex \cite{chazal2008towards}. Another complex, the \emph{Vietoris Rips} (VR) complex starts with $n$ points and $k-$simplices are built on tuples of $k+1$ points whenever all pairwise distances in the tuple are within distance $r$. The key idea of persistent homology is to consider all values of $r$, instead of a fixed choice, by building a nested sequence of simplicial complexes indexed by $r$, called a filtration. 
Applying the $i$th homology with coefficients in a field to a filtration gives a \emph{persistence module}, a functor $(\mathbb{R}, \leq) \to \mathrm{Vect_\mathbb{F}}$.
Persistence modules considered in this paper decompose uniquely up to isomorphism into a direct sum of indecomposable  modules. Indexing this direct sum is a multiset of intervals in the real line called a \emph{barcode} \cite{zomorodian2004computing}.

Here, we study the spaces of point clouds with the same barcode. Taking either the \v{C}ech or Vietoris Rips filtration, the persistence map assigns a point cloud $P$ to its barcode $D_i$ in homological degree~$i$. Hence from our filtration of $P$ we get an ordered collection of barcodes $D=[D_0,D_1,D_2,\ldots]$. 
We define $\PHvr$ to be the map which assigns to a point cloud $P$ its collection of barcodes $D$ under the Vietoris-Rips  filtration. Similarly we define  $\PHc$ to be the analogous map for the \v{C}ech filtration. The dimension of the spaces $(\PH^{\VR})^{-1}(D)$ and $(\PH^{\C})^{-1}(D)$ can be thought of as the maximal number of independent ways a point cloud can be perturbed without affecting its persistent homology, giving us a lower bound. 

In a small neighborhood of any generic\footnote{The definition of \emph{generic} in this paper is of an algebraic nature, and will be detailed in Section \ref{sec:algebra}.} $P$, the map $\PHvr$ is a polynomial function and the map $\PHc$ is a rational function. We can compute the rank of the Jacobian of these maps at $P$.
By genericty of $P$, the Jacobian test says that the rank is the local dimension of the image around the barcode and by standard dimension theory, this gives 
the local dimension of the fiber. The dimension of the whole fiber has to be larger than the local dimension, which gives a lower bound for the dimension of the whole fiber.

\begin{restatable}{theorem}{dimlow}
\label{thm:lowerbound}
Let $\PH$ denote either $\PHc$ or $\PHvr$. For any generic $P \in \PC$, if there are exactly $k$ distinct bounded interval endpoints in barcodes in $D = \PH(P)$, then $\dim\PH^{-1}(D)\geq nd - k + 1$. 
\end{restatable}

Computing the dimension of the fiber is complicated. The PH map, loosely speaking, behaves as if a bunch of rational maps were glued together, where the codomain changes along the boundaries of glued regions. 
Since the codomain of $\PHc$ and $\PHvr$ is the space of ordered collections of barcodes, not $\mathbb{R}^n$, even showing that the spaces $(\PH^{\VR})^{-1}(D)$ and $(\PH^{\C})^{-1}(D)$ have a well-behaved notion of dimension requires some work. Following an observation of Carrière et al  \cite{carriere2021optimizing}, we establish that these level sets are \emph{semialgebraic} sets, which implies that they have a well behaved dimension. 
Even for a generic point cloud, the fiber contains generic and nongeneric points. For this reason, the action of ambient isometries of $\mathbb{R}^d$ does not produce a fiber bundle structure on all level sets of the persistence map. It is worth mentioning that this lack of a fiber bundle structure makes calculating the dimension of level sets quotiented by isometry considerably more difficult. 
It can be proven that the local dimension of a level set quotiented by isometry at $P$ is bounded above the local dimension of the unquotiented level set at $P$ minus $\frac{d(d+1)}{2}$, since $\frac{d(d+1)}{2}$ is the dimension of the isometry group of \(\mathbb{R}^d\). Given these challenges, there is no hope to compute even the dimension of the fiber unquotiented by isometry, but we can bound this number from above as well. 

\begin{restatable}{theorem}{dimup}
    For any ordered collection of barcodes $D = [D_0,D_1,D_2\ldots]$, the dimensions of $(\PHc)^{-1}(D)$ and $(\PHvr)^{-1}(D)$ are both less than or equal to $nd - n + 1$, where $n$ is the number of intervals in $D_0$ (and hence the number of points in any point cloud $P$ in the fiber).
\end{restatable}

We next provide a series of combinatorial conditions on $P$ that are sufficient to imply a barcode is most descriptive, i.e. there is only one point cloud up to ambient isometry that can produce that barcode. These results rely on concepts such as identifiability from real and applied algebraic geometry and rigidity theory from algebraic combinatorics and discrete geometry. In applied algebraic geometry, studying the degree of a single map can a difficult problem. 
Central to identifiability is determining when the fiber of a rational map is a singleton, in which case the point in the fiber is called identifiable; or finite, in which case points in the fiber are called locally identifiable \cite{sullivant2023algebraic,hong2020global,dufresne2018geometry}. There are many situations when such fibers are not even locally identifiable, e.g., point clouds in the plane are not identifiable up to 
ambient isometry under the mapping considered in \cite{BK}; the problem of nonidentifiability is nontrivial for low rank tensors, where the map is the projection from the abstract secant variety to the secant variety and the fiber is related to the notion of decomposition locus \cite{galgano2024identifiability,bernardi2024decomposition}. We have described identifiability of a point in the domain, in complex algebraic geometry, if the map under consideration is algebraic, and the domain is irreducible either almost all points are identifiable or are not. However, since the persistence map is more complicated, we do not have a general dichotomy of this type.
Nevertheless, we make progress on when a barcode is most informative by connecting the problem with rigidity theory \cite{connelly2022frameworks}.
Here, rigidity theory is concerned with identifiability of an unknown set of $n$ labeled points from 
the measurement of some $m$ labeled pairwise distances associated with the edges of a graph $G$, which 
is naturally determined by the dimension of the image and the degree of the polynomial map 
$m_G$ from sets of $n$ labeled points in $\RR^d$ to 
$\RR^m$ that records the $m$ squared edge lengths of the graph $G$. 

We denote by $\PC$ the space of (ordered) point clouds of $n$ distinct points in $d$ dimensional Euclidean space. 
We are interested in collections of barcodes $D$ that are as descriptive as possible. Such a collection $D$ satisfies that every point cloud in the fiber of $D$ is related by an isometry. We say a point cloud that maps to such a collection $D$ under $\PHvr$ (resp. $\PHc$) is \emph{identifiable} under Vietoris-Rips (resp. \v{C}ech) persistence.
To give a sufficient condition for when this happens in the Vietoris-Rips setting, we use notions arising in the rigidity theory of frameworks [ref], an area of algebraic combinatorics.
To each point cloud $P$ we associate a graph $G_P$, for the Vietoris--Rips filtration or 
a hypergraph $H_P$, for the Čech filtration. The pair $(G_\VR,P)$ is called a framework. Intuitively, we should think of an edge in $G_P$ between $p_i$ and $p_j$ as constraining the distance between $p_i$ and $p_j$ to be constant if we allowed the framework $(G_\VR,P)$ to move freely. We say that a 
framework is \emph{globally rigid} $(G_P,P)$ if, whenever $(G_P,Q)$ is another framework with the same edge lengths, 
then $P$ and $Q$ are related by a Euclidean isometry that maps $p_i$ to $q_i$ for all $i$.  A framework is 
\emph{locally rigid} (or simply \emph{rigid}) if there is a neighborhood $U\ni P$ so that if $Q\in U$ and 
the frameworks $(G_P, P)$ and $(G_P,Q)$ have the same edge lengths, they are similarly related by an isometry.  (See 
Sections \ref{sec:rigidity} and \ref{sec:vrrigid} for rigorous definitions).

We can now state one of our main theorem regarding Vietoris-Rips persistence:

\begin{restatable}{theorem}{vrggr}
\label{thm:globrig}
Let $n \geq d+2$ and $d\geq 2$. For generic $P\in\PC$, if $(G_P, P)$ is globally rigid then $P$ is identifiable up to isometry under Vietoris-Rips persistence.
\end{restatable}

A crucial challenge for proving the above result is that, in general, varying $P$ in the fiber of $D$ changes the graph $G_P$. One of the key steps in our proof of Theorem \ref{thm:globrig} is showing that this generically does not happen when $G_P$ is globally rigid. To show this, we leverage a modern rigidity-theoretic result from \cite{gortler2019generic}. As we will discuss in Section \ref{sec:rigidity}, for generic $P$ the global rigidity of $(G_P,P)$ only depends on the combinatorial structure of the graph $G_P$. As such, identifiability of a 
generic point cloud $P$ only depends on the combinatorial properties of the invariant $G_P$.

 To analyze \v{C}ech persistence, a new theory of rigidity is required. We develop this theory in Section \ref{sec:circrigid}, to the point where we can achieve a local, but not global, identifiability result for \v{C}ech persistence. Explicitly, we assign to each point cloud $P$ a hypergraph $H_P$ and the pair $(H_P,P)$ will be an example of what we call a \emph{circumsphere framework}. We say that $(H_P,P)$ is \emph{(locally) rigid} if the only way to perturb $P$ without changing the circumradius of any tuple of points corresponding to a hyperedge in $H_P$ is via an isometry. We say that $P$ is \emph{locally identifiable (under \v{C}ech persistence)} when there is a neighborhood $U$ of $P$ such that for any $Q$ in $U$ with the same barcodes from \v{C}ech persistence, $P$ and $Q$ are isometric. 
 
 We have the following theorem for \v{C}ech persistence:

\begin{restatable}{theorem}{cechglr}
    \label{thm:cechglr}
    A generic $P \in \PC$ is locally identifiable up to isometry under \v{C}ech persistence if and only if the circumsphere framework $(H_P, P)$ is rigid.
\end{restatable}

We give an analogous definition of local identifiability under Vietoris-Rips persistence, which gives rise to an analogous result for Vietoris-Rips persistence, stated in Theorem \ref{thm:vrrigid}.

The rigidity of $(G_P,P)$ and $(H_P,P)$ in fact only depends on the combinatorial structure of $G_P$ or $H_P$ for generic point clouds $P$. For $G_P$ this is a well-known fact from rigidity theory that we discuss in Section \ref{sec:rigidity}. For hypergraphs $H_P$, this fact requires proof, and we establish it in Lemma \ref{lem: circumsphere generic}.

\subsection{Related Work}

Gameiro, Hiraoka and Obayashi provide one of the first results towards inverse problems; namely, if a collection of barcodes $D$ can arise from a point cloud, they provide a point cloud whose persistent homology is $D$, provided that $D$ is close to another collection of barcodes known to come from a point cloud $P$ \cite{gameiro2016continuation}. Oudot and Soloman posed the first problem about level sets of the persistence maps $\PHc$ and $\PHvr$, they asked when they are empty \cite{oudot2020inverse}. While this question is still open, in fact very little is known about the structure of these level sets when they are nonempty.
It is well known that $\PHvr$ and $\PHc$ are never unique due to \emph{isometry invariance}, i.e. Vietoris-Rips filtrations of any isometric point clouds have the same barcodes, and the same is true for \v{C}ech filtrations. 
Beyond this, Smith and Kurlin have given conditions for different point clouds to have identical barcodes for degree 1 homology \cite{smith2022families}. To our knowledge, this is all that is known about the fibers, i.e., the level sets $(\PHvr)^{-1}(D)$ and $(\PHc)^{-1}(D)$, of the persistence map on point clouds. In \cite{lim2024vietoris}, the authors study to what extent the VR barcode of a compact, connected, smooth manifold $M$ determines the geometry of $M$.  By relating the interval in the VR barcode that corresponds to the fundamental class of M to the filling radius, they obtain various sufficient conditions for the barcode to determine $M$ in the case where $M$ is a sphere. There is also a thread of research studying level sets of persistent homology applied to filtrations arising from functions, not point clouds, see \cite{curry2018fiber,cyranka2020contractibility,mischaikow2021persistent,leygonie2022fiber, leygonie2023fiber,leygonie2024algorithmic, beers2023fiber}.

\subsection{Plan of the Paper}
Section \ref{sec:prelims} reviews the notions from topology, geometry, algebra, combinatorics, graph theory and rigidity theory that we will need throughout the paper. In Section \ref{sec:ph} we discuss how persistent homology assigns barcodes to arbitrary filtrations of spaces, specializing to filtrations arising from point clouds in Section \ref{sec:filt}. Then in Section \ref{sec:spanningtrees} we discuss some results in TDA which are well known but which we could not find proven in detail in the literature. After this, we move to the prerequisite notions we need in algebra and geometry. First, we define algebraic and semialgebraic sets, and show how these notions lead to definitions of dimension and genericity in Section \ref{sec:algebra}. Then we define circumspheres and enclosing spheres of point clouds and discuss important facts we will need about these kinds of spheres in Section \ref{sec:circumsphere}. This section will in particular be useful for studying \v{C}ech persistence. We conclude our section on preliminaries with a review of rigidity theory in Section \ref{sec:rigidity}.

In Section \ref{sec:semialgebraic} we give a detailed semialgebraic setup for the PH map. One of the main results in this section is how to chop the domain of the persistence map up into semialgebraic cells on which PH is better behaved. We then establish facts about these cells that we use throughout the paper. Following, Sections \ref{sec:upperbound} and \ref{sec:lowerbound} give an upper and lower bound on the dimension of the fiber of the persistence map.

In Section \ref{sec:vrrigid} we define the critical graph $G_P$ of $P$ as the graph defined by the set of edges that appear in the Vietoris-Rips filtration of $P$ at a value of a barcode endpoint. This leads us to our main identifiability result, Theorem \ref{thm:globrig}, along with our local identifiability result for Vietoris-Rips persistence, Theorem \ref{thm:vrrigid}. In Section \ref{sec:circrigid} we develop a new rigidity theory. This allows for us to prove our local identifiability result for \v{C}ech persistence, Theorem \ref{thm:cechglr}, in Section \ref{sec:cechrigid}.

We conclude with Section \ref{sec:open}, where we discuss a few open questions arising from our work.

In the appendix we prove some elementary and well known facts about degree zero persistence for point clouds and provide references for a couple of basic, though somewhat technical, lemmas about semialgebraic sets.

\section{Preliminaries}
\label{sec:prelims}

\subsection{Persistent Homology}
\label{sec:ph}
Persistent homology is an operation which takes as input a nested collection of topological spaces called a filtration:

\begin{definition}
    A \emph{filtration} is a collection of topological spaces $\{X_t\}_{t\in \mathbb{R}}$ such that $X_s \subseteq X_t$ whenever $s \leq t$.
\end{definition}

In this paper, $H_i(\bullet)$ will always denote homology in degree $i$ over a field $\mathbb{F}$. All statements hereafter will be true regardless of the choice of $\mathbb{F}$. When we apply $H_i(\bullet)$ to the spaces in a filtration $\{X_t\}_{t\in \mathbb{R}}$ we obtain a collection of vector spaces $\{H_i(X_t)\}_{t\in\mathbb{R}}$. Moreover, the functoriality of homology implies that we have commuting linear maps $H_i(X_s) \to H_i(X_t)$ whenever $s \leq t$. This is an example of an algebraic structure called a persistence module.

\begin{definition}
    A \emph{persistence module} is a collection of vector spaces $V = \{V_t\}_{t\in\mathbb{R}}$ equipped with commuting linear maps $V_{s,t}: V_s \to V_t$ for all $s \leq t$ satisfying $V_{s,s} = \mathrm{id}$. Two persistence modules $V$ and $W$ are said to be \emph{isomorphic} if there are vector space isomorphisms $\phi_t: V_t \to W_t$ such that $\phi_t \circ V_{s,t} = W_{s,t} \circ \phi_s$.

    The persistence module $V$ is called \emph{pointwise finite dimensional} (or \emph{pfd}) if $V_t$ is finite for each $t\in \mathbb{R}$. A number $t\in \mathbb{R}$ is called a \emph{critical value} of $V$ if there does not exist an $\epsilon > 0$ such that $V_{r,s}$ is an isomorphism for all $t-\epsilon< r\leq s < t+\epsilon$. A persistence module is called \emph{tame} if it is pfd and has finitely many critical values. 
\end{definition}

We are interested in tame persistence modules for two reasons. First, they frequently arise when studying finite point clouds as we will see shortly. Second, they admit decompositions into basic persistence modules called interval modules.

\begin{definition}
    Given an interval $I\subseteq \mathbb{R}$, define the persistence module $\chi_I$ by 
    
    \begin{equation*}
    (\chi_I)_t = \begin{cases}
            \mathbb{F} & t\in I \\
            0 & \textrm{otherwise,}
        \end{cases} \qquad
    (\chi_I)_{s,t} = \begin{cases}
            \mathrm{id} & s,t\in I \\
            0 & \textrm{otherwise,}
        \end{cases}
\end{equation*}
Such a persistence module is called an \emph{interval module}.
\end{definition}

For persistence modules $V$ and $W$, the direct sum $V \oplus W$ is defined by $(V \oplus W)_t := V_t \oplus W_t$, with $(V\oplus W)_{s,t}$ being the obvious maps inherited from $V_{s,t}$ and $W_{s,t}$. The following result, which follows from Theorem 4.6 and Corollary 4.7 of \cite{bubenik2014categorification} for example, shows that any tame persistence module decomposes into a direct sum of interval modules. 

\begin{theorem}
\label{thm:pmdecomp}
If $V$ is a tame persistence module, there is a unique multiset of nonempty intervals $\mathcal{I}$ such that $V$ is isomorphic to
\begin{equation*}
    \bigoplus_{I\in \mathcal{I}} \chi_I,
\end{equation*}
and the multiset $\mathcal{I}$ is finite.
\end{theorem}

We remark that this result also follows fairly directly from the classic \cite{zomorodian2004computing}, and the from the more general main result of \cite{crawley2015decomposition}. However the statement of the decomposition theorem given in \cite{bubenik2014categorification} is most convenient for our purposes. The above theorem motivates the following definition.

\begin{definition}
    A \emph{barcode} is a multiset of nonempty intervals in the real line. A \emph{full barcode} is a sequence of barcodes $D = \{D_i\}_{i = 0}^{\infty}$.
\end{definition}

Using Theorem \ref{thm:pmdecomp} we may finally arrive at the definition of persistent homology.

\begin{definition} 
    For a filtration $\mathcal{X} = \{X_t\}_{t\in\mathbb{R}}$, let $H_i(\mathcal{X})$ denote the persistence module $\{H_i(X_t)\}_{t\in\mathbb{R}}$, and suppose this persistence module is tame. The \emph{persistent homology in degree $i$} of $\mathcal{X}$ is the barcode $D_i$ satisfying
    \begin{equation*}
        H_i(\mathcal{X}) \cong \bigoplus_{I\in D_i} \chi_I.
    \end{equation*}
    If $H_i(\mathcal{X})$ is tame for each $i\geq 0$, the \emph{persistent homology} of $\mathcal{X}$ is the full barcode $D = \{D_i\}_{i = 1}^{\infty}$.
\end{definition}

\subsection{Filtrations for Point Clouds}
Point clouds are the central objects of interest to this paper, and we define them as follows.

\begin{definition}
We define $\PC$ as
\begin{equation*}
    \PC = \{(p_1,\ldots, p_n)\in \mathbb{R}^{nd} : \; p_i \in \mathbb{R}^d \text{ and } p_i \neq p_j\; \forall i \neq j\},
\end{equation*}
the configuration space of $n$ points in $\mathbb{R}^d$. Any element of $\PC$ is called a \emph{point cloud}\footnote{Strictly speaking $\PC$ is the space of \emph{ordered} point clouds of $n$ points in $\mathbb{R}^d$. By quotienting by the action of the symmetric group $S_n$, acting on $\PC$ by permuting the order of the $p_i$, we may obtain the space of unordered point clouds of $n$ points in $\mathbb{R}^d$, $\PC/S_n$.} (of $n$ points in dimension $d$).
\end{definition}

\label{sec:filt}
Point clouds give rise to filtrations, and these filtrations allow us to assign barcodes to point clouds. The following is a kind of filtration arising from point clouds.

\begin{definition}
    For $x\in \mathbb{R}^d$, led $B_r(x)$ denote the closed ball of radius $r$ centered at $x$. Given $P = (p_1,\ldots,p_n) \in \PC$ the \emph{\v{C}ech filtration of $P$} is the filtration $\{P_r\}_{r\in\mathbb{R}}$, where
    \begin{equation*}
        P_r = \bigcup_{i=1}^n B_r(p_i).
    \end{equation*}
    We denote the persistent homology in degree $i$ of this filtration by $\PHc_i(P)$, should it be well defined, and the persistent homology of this filtration by $\PHc(P)$, should it well defined.
\end{definition}

We will see that $\PHc_i(P)$ and $\PHc(P)$ always exist by using a combinatorial reformulation of the \v{C}ech filtration of a point cloud, but before we introduce it, we need the following definition.

\begin{definition}
Let $K_0$ be a finite set.  An (abstract) \emph{simplicial complex} $K$ is a set of nonempty subsets $\sigma \subseteq K_0$ such that if $\tau \subseteq \sigma$ and $\sigma\in K$, then $\tau\in K$.  An element $\sigma \in K$ is called a \emph{simplex} (pl. simplices),  
and it is called a $p$-simplex if $\sigma$ has $p+1$ elements exactly. The $p$-skeleton of $K$, denoted $K_p$
is the simplicial complex consisting of all $j$-simplices of $K$, for $j\leq p$.  The \emph{dimension} of $K$ is the 
maximum $d$ such that $K$ contains a $d$-simplex.

A $d$-dimensional simplicial complex $K$ has a \emph{geometric realization}.  Let $f: K_0\to \RR^{2d + 1}$ be an injective mapping 
from $K_0$ to a point set such that any $2d +2$ of the points in $f(K_0)$ are affinely independent. The geometric realization of $K$ is 
\[
    \|K\| = \bigcup_{\sigma\in K} \operatorname{conv} \{f(v) : v\in K_0\cap \sigma\}
\]
with the subspace topology from the ambient $\RR^{2d + 1}$.  (All the choices of $f$ give an equivalent space up 
to a PL homeo, for example.)\footnote{This construction is outlined in detail in \cite[Chapter 3.1]{edelsbrunner2022computational}}
\end{definition}

Note that the 1-skeleton of a simplicial complex is a graph. It is well known that at the level of homology, the \v{C}ech filtration is a filtration of simplicial complexes. For convenience we will let $[n]$ denote the set $\{1,\ldots, n\}$ throughout. We will also denote by $|K|$ the number of simplices in $K$. Of particular interest to us will be the simplicial complex whose simplices are the nonempty subsets of $[n]$. This simplicial complex is called the complete simplicial complex on $n$ vertices. We denote this simplicial complex by $K(n)$.

\begin{definition}
For $P = (p_1,\ldots, p_n) \in \PC$, let $\C(P,r)$ denote the simplicial complex
\begin{equation*}
    \{\sigma \subseteq [n] : \text{there exists } q\in \mathbb{R}^d \text{ such that } d(p_i,q) \leq r  \text{ for all } i \in \sigma\}.
\end{equation*}

These simplicial complexes define a filtration $\{\C(P,r)\}_{r\in \mathbb{R}}$.
\end{definition}

The following, which is a result of \cite[Section 3.1]{chazal2008towards}, shows that indeed the filtration $\{\C(P,r)\}_{r\in \mathbb{R}}$ of a point cloud has the same persistent homology as its \v{C}ech filtration.

\begin{proposition}
    \label{prop:nerveiso}
    The persistence modules $\{H_i(P_r)\}_{r \in \mathbb{R}}$ and $\{H_i(\C(P,r))\}_{r \in \mathbb{R}}$ are isomorphic for each $i\geq 0$.
\end{proposition}

It is easily observed that $\{H_i(\C(P,r))\}_{r\in\mathbb{R}}$ is tame, so $\PHc_i(P)$ and $\PHc(P)$ are defined for all $P\in \PC$ by Proposition \ref{prop:nerveiso}.

The Vietoris-Rips filtration is another commonly used filtration which is also defined simplicially:

\begin{definition}
    For $P = (p_1,\ldots, p_n) \in \PC$, let $\VR(P,r)$ denote the simplicial complex
    \begin{equation*}
        \{\sigma\subseteq [n]: d(p_i,p_j)\leq 2r \text{ for all }i,j\in\sigma\}.
    \end{equation*}
    We refer to the filtration $\{\VR(P,r)\}_{r\in\mathbb{R}}$ as the \emph{Vietoris-Rips filtration of $P$}. We denote the persistent homology in degree $i$ of this filtration by $\PHvr_i(P)$, and the persistent homology of this filtration by $\PHvr(P)$.
\end{definition}

Clearly $\PHvr_i(P)$ and $\PHvr(P)$ are well defined since $\{H_i(\VR(P,r))\}_{r\in\mathbb{R}}$ is tame. It is easily checked that the \v{C}ech and Vietoris-Rips filtrations have identical 1-skeletons. For this reason, it is an immediate consequence of the theory of simplicial homology that $\PHc_0(P) = \PHvr_0(P)$ for any point cloud $P$. We note that $\PHvr(P)$ and $\PHc(P)$ are in general \emph{not} identical. The advantage of $\PHc$ is that it reflects the changing topology of the spaces $P_r$, whereas the advantage of $\PHvr$ is that it is easier to compute for high dimensional point clouds \cite{bauer2021ripser}.

We conclude the subsection with a discussion of how to use a function on a simplicial complex to get a filtration. We will identify maps $f:K \to \mathbb{R}$ with elements of $\mathbb{R}^{|K|}$, where the $\sigma^{\mathrm{th}}$ coordinate of $f$ when viewed this way, denoted $f_\sigma$, is the number $f(\sigma)$. If $f$ satisfies that $f_\tau \leq f_\sigma$ whenever $\tau\subseteq\sigma$, then it follows that the sets $\{\sigma \in K:f_\sigma \leq t\}_{t\in\mathbb{R}}$ are simplicial complexes, motivating the following definition.

\begin{definition}
    Let $K$ be a simplicial complex with $f \in \mathbb{R}^{|K|}$ such that $f_\tau \leq f_\sigma$ whenever $\tau\subseteq\sigma$. Then $f$ is called an \emph{order preserving map} and we denote by $\PH_i(f)$ the persistent homology in degree $i$ of the filtration of simplicial complexes
    \begin{equation*}
        \{\sigma \in K:f_\sigma \leq t\}_{t\in\mathbb{R}}.
    \end{equation*}
    We denote by $\PH(f)$ the persistent homology of this filtration.

    A map $\Phi:\PC\to\mathbb{R}^{|K|}$ is called an \emph{order preserving map parametrized by $\PC$} if $\Phi(P)$ is an order preserving map for each $P$.
\end{definition}
For example, defining $\Phi^{\VR}:\PC \to \mathbb{R}^{|K(n)|}$ by \begin{equation*}
    \Phi^{\VR}_\sigma(P) := \max_{i,j\in\sigma}\frac{1}{2}d(p_i,p_j),
\end{equation*}
then we have $\PH(\Phi^{\VR}(P)) = \PH^{\VR}(P)$.

Similarly, defining $\Phi^{\C}:\PC \to \mathbb{R}^{|K(n)|}$ by \begin{equation*}
    \Phi^{\C}_\sigma(P) := \min\Big\{r\in \mathbb{R}: \text{ there exists } q\in \mathbb{R}^d \text{ such that } d(p_i,q) \leq r  \text{ for all } i \in \sigma\Big\},
\end{equation*}
then we have $\PH(\Phi^{\C}(P)) = \PH^{\C}(P)$. In Section \ref{sec:semialgebraic} we will prove a general structure theorem about the following kinds of sets
\begin{align*}
    (\PH_i\circ\Phi)^{-1}(B) &:= \{P \in \PC: \PH_i(\Phi(P)) = B\}\\
    (\PH\circ\Phi)^{-1}(D) &:= \{P \in \PC: \PH(\Phi(P)) = D \},
\end{align*}
where $\Phi:\PC\to\mathbb{R}^{|K|}$ is an order preserving map on a simplicial complex $K$ parametrized by $\PC$, $B$ is a barcode, and $D$ is a full barcode.

\subsection{Persistence and Minimal Spanning Trees}
\label{sec:spanningtrees}

The goal of this section is to state two results in topological data analysis which are fairly elementary, we believe to be widely known to be true, that we will need later. Proofs of results in this subsection are given in the appendix for completeness. First we establish a basic fact about degree zero \v{C}ech and Vietoris-Rips barcodes. 

\begin{restatable}{lemma}{hzerofacts}
    \label{lem:basicH0}
    Let $P = \{p_1,\ldots,p_n\}\in\PC$. Let $D_0 = \PHc_0(P)$  (or equivalently $D_0 = \PHvr_0(P)$). The barcode $D_0$ has $n$ intervals, each with a closed left endpoint and open right endpoint. All of these intervals have left endpoint equal to 0. One of these intervals has $\infty$ as its right endpoint. All other intervals in $D_0$ have a positive right endpoint.
\end{restatable}

The next result discussed here (mentioned in \cite{cultrera2024chromatic}, and \cite{elkin2020mergegram} in a slightly different context, for example) is widely known, but a little less clear. We will state it using the following language.

\begin{definition}
    A \emph{tree} is a connected cycle free graph. Let $P = \{p_1,\ldots,p_n\}\in\PC$. A \emph{minimal spanning tree} of $P$ is a tree whose vertex set is $[n]$ and whose edge set $E$ satisfies
    \begin{equation*}
        \sum_{\{i,j\}\in E}d(p_i,p_j) = \min \Big\{\sum_{\{i,j\}\in E'}d(p_i,p_j): [n] \text{ and } E' \text{are the vertices and edges of a tree.}\Big\}.
    \end{equation*}
\end{definition}

The intuition behind this definition is that a minimal spanning tree has the smallest cumulative length of edges while still spanning the metric space $P$. The result we need says that the edge lengths of a minimal spanning tree are specified by barcodes arising from $P$.

\begin{restatable}{lemma}{minspantree}
\label{lem:minspantree}
    Let $P = \{p_1,\ldots,p_n\}\in\PC$. Let $D_0 = \PHc_0(P)$  (or equivalently $D_0 = \PHvr_0(P)$). Let $r_1<\ldots<r_m$ denote the distinct finite right endpoints of intervals in $D_0$ and let $\mu_i$ denote the multiplicity with which $r_i$ appears as a right endpoint in $D_0$. Let $T$ be a minimal spanning tree of $P$. Then the multiset of edge lengths of $T$ consists of the numbers $2r_1,\ldots,2r_m$ where $2r_i$ appears with multiplicity $\mu_i$.
\end{restatable}

It is explained how this fact follows from Kruskal's algorithm in \cite[Section 2.3]{cultrera2024chromatic}. For completeness, we nevertheless give a detailed proof of the lemma in the appendix.

\subsection{Semialgebraic sets, dimension, and genericity}
\label{sec:algebra}

Throughout we will consider only polynomials with real coefficients and work with the following kinds of sets.
\label{sec:dimdef}
\begin{definition}
    A set $S \subseteq\mathbb{R}^m$ is called \emph{algebraic} if $S$ is of the form
    \begin{equation*}
        \{x \in \mathbb{R}^m:f_1(x) = \ldots = f_k(x) = 0\},
    \end{equation*}
    where each $f_i$ is a polynomial in $x$.
    
    A set $S\subseteq \mathbb{R}^m$ is called \emph{semialgebraic} if $S$ is a finite of union of sets of the form
    \begin{equation}
    \label{eqn:basicsemialg}
        \{x \in \mathbb{R}^m: f_1(x) = \ldots = f_k(x) = 0,\; g_1(x)>0,\ldots,g_l(x)> 0\},
    \end{equation}
    where each $f_i$ and $g_i$ is a polynomial in $x$.

    If $S\subseteq \mathbb{R}^m$, function $f:S \to \mathbb{R}^k$ is called semialgebraic if its graph, viewed as a subset of $\mathbb{R}^{m+k}$ is semialgebraic.

    The \emph{Zariski closure} of a set $S\subseteq \mathbb{R}^m$ is the intersection of all algebraic sets containing $S$.
\end{definition}

One can show that the Zariski closure of any set is algebraic.

It can be checked also that finite unions of sets as in Equation \ref{eqn:basicsemialg}, but with any number of the $g_i(x)>0$ replaced with $g_i(x) \neq 0$ or $g_i(x) \geq 0$, are also semialgebraic. It follows that the intersection of any semialgebraic set with $\PC$ is semialgebraic. We state some basic results about semialgebraic sets, all of which are discussed in detail in \cite{bochnak2013real}.

Semialgebraic sets have the following basic properties
\begin{enumerate}
    \item Finite unions, finite intersections, and complements of semialgebraic sets are semialgebraic.
    \item If $A\subseteq \mathbb{R}^m$ and $B \subseteq \mathbb{R}^n$ are semialgebraic, so is $A \times B \subseteq \mathbb{R}^{m+n}$.
    \item If $\pi:\mathbb{R}^{d+1} \to \mathbb{R}^d$ is an axis-aligned projection map, $\pi$ sends semialgebraic sets to semialgebraic sets\footnote{This result is significantly more challenging to show than others in this list and is known as the Tarski-Seidenberg theorem, see for example \cite[Proposition 5.2.1]{bochnak2013real}}.
\end{enumerate}

These facts are powerful tools for proving that a set is semialgebraic. Here is a simple but important example of how these facts may be used to show a set is semialgebraic.

\begin{example}
    Any rational function $f:A\to \mathbb{R}$ is semialgebraic, provided the domain $A$ of $f$ is also semialgebraic. To see this, write $f(x) = g(x)/h(x)$ for $g, h$ polynomials. The graph of $f$ is the set
    \begin{equation*}
        \Gamma = \{(x,y) \in A \times \mathbb{R}: yh(x) = g(x)\}.
    \end{equation*}
    If $A\subseteq \mathbb{R}^m$, then $\Gamma$ is the intersection of semialgebraic sets $A \times \mathbb{R}$ and $\{(x,y) \in \mathbb{R}^{m+1}: yh(x) - g(x) = 0\}$. So $\Gamma$ is semialgebraic and therefore $f$ is semialgebraic as well.
\end{example}

Whats more, one of the properties that semialgebraic sets enjoy as a consequence of the three above facts is that any set that can defined by a sentence in first order logic with polynomial equalities and inequalities is semialgebraic. The process of constructing the semialgebraic set corresponding to such a sentence is called \emph{quantifier elimination}. To show the reader why and how this works, we present a simple example of quantifier elimination in action, and refer the reader to \cite{bochnak2013real} for the general details.

\begin{example}
    Let $f = (f_1,f_2):\mathbb{R} \to \mathbb{R}^2$ be a polynomial map. Then the set of points not in the image of $f$, i.e. the set 
    \begin{equation*}
        S : =\{(x_1,x_2)\in \mathbb{R}^2 : (\neg \exists \, y :(f_1(y) = x_1))\wedge (f_2(y) = x))\}
    \end{equation*}
    is semialgebraic.

    We can get rid of the `not' symbol $\neg$ in $S$ by taking a complement, i.e. we have $S = \mathbb{R}^2 - X$. Where $X:=\{(y_1,y_2)\in \mathbb{R}^2 : ( \exists \, y :(f_1(x) = y_1))\wedge (f_2(x) = y_2))\}$. If $X$ is semialgebraic, then so is $S$, since semialgebraic sets are closed under relative complements. Now we let
    \begin{equation*}
    Y:= \{(x,y_1,y_2)\in \mathbb{R}^3 : (f_1(x) = y_1))\wedge (f_2(x) = y_2))\}.    
    \end{equation*}
    It suffices to show $Y$ is semialgebraic since $X$ is a projection of $Y$. Finally, $Y$ is the intersection of $Z_1:=\{(x,y_1,y_2)\in \mathbb{R}^3: f_1(x) = y_1\}$ and $Z_2:=\{(x,y_1,y_2)\in \mathbb{R}^3: f_2(x) = y_2\}$, and so is semialgebraic, since $Z_1$ and $Z_2$ are semialgebraic (and in fact algebraic). Therefore $S$ is semialgebraic.
\end{example}

The topological interior, closure, and boundary of a semialgebraic set (with respect to the Euclidean topology) are all semialgebraic. Semialgebraic sets can be partitioned into finitely many sets, each homeomorphic to $(0,1)^k$, for $k$ potentially varying among the sets. In fact, semialgebraic sets admit Whitney stratifications, but we will not need this.

\begin{definition}
    If $S$ is a semialgebraic set partitioned into $r$ sets $S_1,\ldots, S_r$, with $S_i$ homeomorphic to $(0,1)^{k_i}$, then the \emph{dimension} of $S$, denoted $\dim S$, is defined to be $\max_{i}(k_i)$. If $S$ is empty, the dimension of $S$ is defined to be $-\infty$.
\end{definition}

It turns out that the dimension of $S$ does not depend on the choice of partition. Moreover, if $S$ is also a smooth real manifold of dimension $k$, then the manifold dimension of $S$ and the dimension of $S$ as a semialgebraic set agree. As such, the dimension of $\PC$ is $nd$. Furthermore the dimension of $S$ also agrees with its algebraic dimension, i.e. the dimension of its Zariski closure.
The dimension of a finite union of sets $X_i$ is the maximum of the dimension of $X_i$ over the finitely many indices $i$.

\begin{definition}\label{def: generic property}
    We say that a property $\mathcal{Q}$ is \emph{generic}
    if there is a proper algebraic subset $A\subset \mathbb{R}^{nd}$ such that 
    either the property $\mathcal{Q}$ holds for all $P \in \PC - A$ or the property $\mathcal{Q}$ does not hold for all $P \in \PC - A$. If we are in the first situation we say $\mathcal{Q}$ is \emph{generically true}, or that, \emph{for generic point clouds $P$}, $\mathcal{Q}$ holds.
\end{definition}

We remark that proper algebraic subsets of $\mathbb{R}^n$ in particular are nowhere dense and have measure zero.

To show that a property $\mathcal{Q}$ holds for all generic point clouds it suffices to find a \emph{semialgebraic} subset $A\subseteq \PC$ of dimension less than $nd$ such that $\mathcal{Q}$ holds for all $P \in \PC - A$. This is true because the Zariski closure of $A$ is algebraic, contains $A$, and has dimension less than $nd$, and therefore is a proper algebraic subset of $\mathbb{R}^{nd}$. We will use this fact repeatedly throughout the paper.

\subsection{Enclosing Spheres and Circumspheres}
\label{sec:circumsphere}

Here we recall well known facts about enclosing spheres and circumspheres of collections of points. These facts will help us establish a rigidity theory related to \v{C}ech filtrations in Sections \ref{sec:circrigid}-\ref{sec:cechrigid} and, setting aside Lemma \ref{lem:semialgenclosing}, will not be used elsewhere. 

For convenience, we will always consider a sphere of radius $0$ centered at $c\in \mathbb{R}^d$ to be the set $\{c\}$.

\begin{definition}
    Let $P= (p_1,\ldots,p_n) \in \PC$.
    Let $S$ be a sphere of dimension $d-1$, with any radius and center. If $S$ intersects every point in $P$, then $S$ is called a \emph{circumsphere of $P$}. If $S$ is a circumsphere of $P$ and no circumsphere of $P$ exists with a smaller radius, then $S$ is called the \emph{minimal circumsphere of $P$}. The radius of the smallest circumsphere of $P$ is called the \emph{circumradius of $P$}.

    A sphere whose associated disk contains $P$ is called an \emph{enclosing sphere of $P$}. The smallest sphere that encloses $P$ is called the \emph{minimal enclosing sphere of $P$}. We refer to the radius of this sphere as the \emph{enclosing radius of $P$}.
\end{definition}

We illustrate these kinds of spheres in Figure \ref{fig:spheres}. 

\begin{figure}[htpb]
    \centering
    \includegraphics[width=0.6\textwidth]{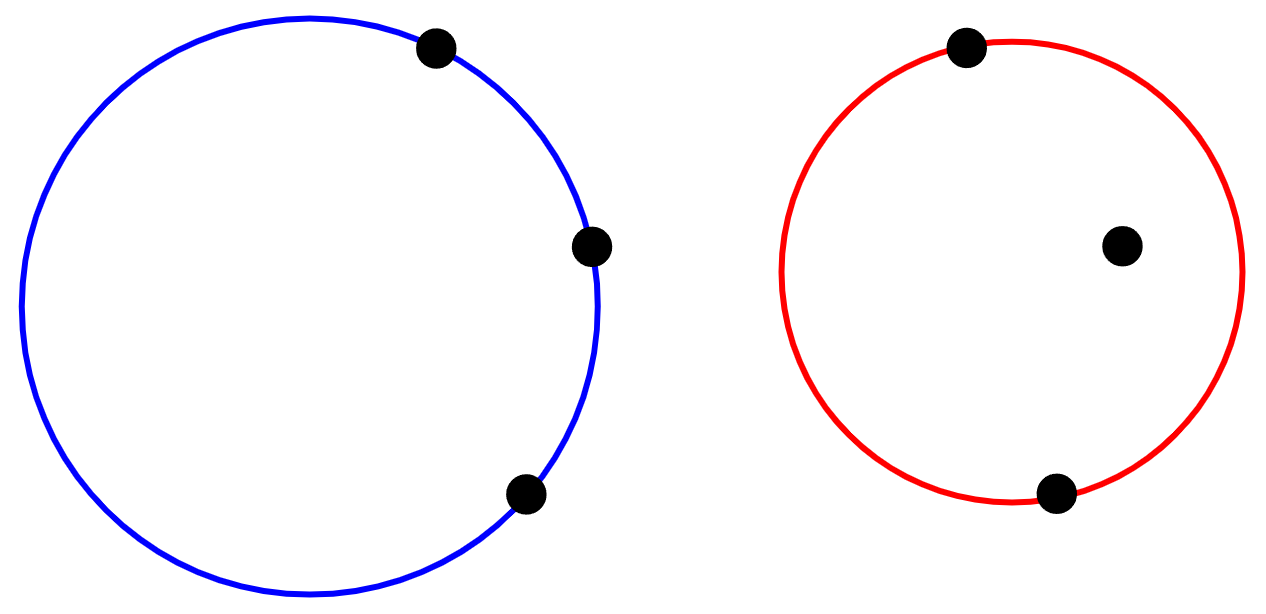}
    \caption{On the left is a circumsphere of a planar point cloud. In fact, since this point cloud has three points, it has at most one circumsphere, so this circumsphere is minimal. On the right is the minimal enclosing sphere of the same point cloud. Note that this enclosing sphere is the minimal circumsphere of a sub-point cloud.}
    \label{fig:spheres}
\end{figure}

The following lemma is a consequence of \cite[Lemma 1]{welzl2005smallest}\footnote{While this result is only stated for the $k=2$. It is pointed out later in the paper that that the same lemma and proof holds for arbitrary $k$, with every instance of the number 3 replaced by $k+1$.}.

\begin{lemma}[\cite{welzl2005smallest}]
    \label{lemma:sphereunique}
    Let $P\in \PC$ be a point cloud with affine span of dimension $l$. If a circumsphere of $P$ exists, then the minimal circumsphere exists and is unique.
    
    Regardless of the choice of $P$, the minimal enclosing sphere of $P$ exists and is unique. 
    Moreover there is a set $\tau = \{i_1,\ldots, i_k\} \subseteq [n]$ where $k\leq l+1$ and the following are equal
    \begin{enumerate}
        \item The minimal enclosing sphere of $P$,
        \item The minimal enclosing sphere of $(p_{i_1},\ldots, p_{i_k})$, and
        \item The minimal circumsphere of $(p_{i_1},\ldots, p_{i_k})$, 
        which must have its center in the affine span of 
            the sub-configuration.
    \end{enumerate}
\end{lemma}

We will also need that the minimal enclosing radius function is reasonably well behaved.

\begin{lemma}
    \label{lem:semialgenclosing}
    The map $\rho_\sigma:\PC \to \mathbb{R}$, sending $P$ to the minimal enclosing radius of $\{p_i: i\in\sigma\}$ is semialgebraic.
\end{lemma}

\begin{proof}
    Let $\sigma = \{i_1,\ldots,i_k\}$. The graph of $\rho_\sigma$ is the set
    \begin{align*}
    \label{eqn:enclosingmindef}
       \bigg\{ (P,r)\in \PC \times \mathbb{R} : \Big(\exists x: &\big(d(x,p_{i_1}) \leq r \wedge \ldots \wedge d(x,p_{i_k}) \leq r\big)\Big)\\
        &\wedge \Big(\neg\exists r': \exists x: \big(r' < r \wedge d(x,p_{i_1}) \leq r' \wedge \ldots \wedge d(x,p_{i_k}) \leq r'\big)\Big)\bigg\}.
    \end{align*}
    In the above formula, the first line says that $r$ is large enough furnish an enclosing sphere of $\sigma(P)$, and the second line says that moreover, no smaller $r'<r$ has this property, so that $r$ is the minimal enclosing radius of $\sigma(P)$. It follows by quantifier elimination that the graph of $\rho_\sigma$, and hence $\rho_\sigma$ itself, is semialgebraic.
\end{proof}

We note that the above result is already known in the TDA community, see for example \cite{carriere2021optimizing}.

To have a formulaic description of the minimal enclosing radius function we will need two different kinds of matrices.

\begin{definition}
Let $P = (p_1,\ldots,p_n)\in \PC$.  The  \emph{Euclidean distance matrix of $P$}
is the $n\times n$ matrix $\Lambda_P$ that has entries 
\[
    (\Lambda_P)_{ii} = 0 \qquad \text{and}
    \qquad (\Lambda_P)_{ij} = d(p_i,p_j)^2\quad (\text{when }i\neq j)
\]
The \emph{Cayley-Menger matrix of $P$} is the $(n+1) \times (n+1)$ matrix 
$\Delta_P$
\begin{equation*}
\Delta_P =
\left[ 
\begin{array}{cc} 
  0 & \theta^{T} \\ 
  \theta & \Lambda_P
\end{array} 
\right],
\end{equation*}
where $\theta$ is a column vector of $n$ ones.

For $\sigma = \{i_1,\ldots, i_k\}\subseteq [n]$ we denote by 
$\sigma(P)$ the point cloud $\sigma(P) = (p_{i_1}, \ldots, p_{i_k})$.
\end{definition}

The following result from \cite[Theorem 2.1.3]{fiedler} combined with Lemma \ref{lemma:sphereunique} allows us to compute the radii of enclosing spheres.

\begin{proposition}
    \label{prop:circumradeqn}
    Let $P \in \PC$ and $\tau \subseteq [n]$ be such that $\tau(P)$ is affinely independent. Then the circumradius of $\tau(P)$ is defined and its square is equal to
    \begin{equation}
    \label{eqn:encrad}
        -\frac{\det\Lambda_{\tau(P)}}{2\det\Delta_{\tau(P)}}.
    \end{equation}
\end{proposition}

\subsection{Rigidity Theory}
\label{sec:rigidity}
We now introduce some background on the rigidity theory of frameworks.  For a 
general reference, see, e.g., \cite{connelly-book}.  The basic objects of study 
are frameworks, which, informally are a placement of a graph into a $d$-dimensional 
Euclidean space.
\begin{definition}\label{def: framework}
Let $d\in \NN$. A \emph{$d$-dimensional framework $(G,P)$} is a pair $(G,P)$ where 
$G = (V,E)$ is a simple, undirected graph with vertex set $V = [n]$ and $m$ edges and 
$P = (p_1, \ldots, p_n)$ is a configuration of $n$ points in $\RR^d$.  We treat $G$ 
as an ordered graph, using the ordering on $[n]$ for the vertices and any bijection 
$E\to [m]$ to fix an ordering of the edges.

The squared edge-length measurement map 
\[
f_G : \left((\RR^d)^n\right) \to \RR^m \qquad P \mapsto \left(\|p_i - p_j\|^2\right)_{\{i,j\}\in E}
\]
maps a framework to its vector of edge lengths.  The correspondence between the coordinates 
of $\RR^m$ is from the ordering of the edges.
\end{definition}

In this paper we will be interested in the following notions of equivalence between point clouds.

\begin{definition}
Two point clouds $P = (p_1,\ldots,p_n),Q = (q_1,\ldots,q_n) \in \PC$ are \emph{congruent} if $d(p_i,p_j) = d(q_i,q_j)$ for all $1 \leq i,j\leq n$. The point clouds $P$ and $Q$ are \emph{isometric} if $P$ and $Q$ are isometric as finite metric spaces $\mathbb{R}^d$. Equivalently, $P$ and $Q$ are isometric if there is a permutation $f:[n] \to [n]$ such that $d(p_i,p_j) = d(q_{f(i)},q_{f(j)})$ for each $i,j\in [n]$. The point clouds $P$ and $Q$ are \emph{weakly homometric} if the multiset of pairwise distances $d(p_i,p_j)$ in $P$ is equal to the multiset of pairwise distances of $Q$.
\end{definition}

Note that $P = (p_1,\ldots,p_n)$ and $Q = (q_1,\ldots,q_n)$ are congruent if and only if the map $\phi$ which sends each $p_i$ to $q_i$ is an isometry. As such congruent point clouds are necessarily isometric, but the converse is not true. It follows from the classical work of Young and Householder that if $P$ and $Q$ are congruent, there is an isomorphism $\phi$ of $\mathbb{R}^d$ sending $p_i$ to $q_i$ \cite{young1938discussion}.

If $P = (p_1,\ldots,p_n)$ and $Q = (q_1,\ldots,q_n)$ are merely isometric, then there is a bijection $f:[n] \to [n]$ such that $d(p_i,p_j) = d(q_{f(i)},q_{f(j)})$. Hence $P$ is congruent to $\widetilde{Q} = (q_{f(1)},\ldots,q_{f(n)})$. So by the previous paragraph there is an isometry $\phi$ of $\mathbb{R}^d$ sending $p_i$ to $q_{f(i)}$. We thus observe that isometries of point clouds lift to isometries of $\mathbb{R}^d$.

Isomorphic point clouds are weakly homometric, but the converse is not true, as shown by the two non-isomorphic point clouds in Figure \ref{fig:genericneeded}. The notion of weakly homometric point clouds is not as important to this paper as that of isometry and congruence. Our reason for defining this notion here is to emphasize that, a point cloud is not always recovered up to isometry from its pairwise distances alone.

\begin{figure}[htbp]
    \centering
    \includegraphics[width=0.8\textwidth]{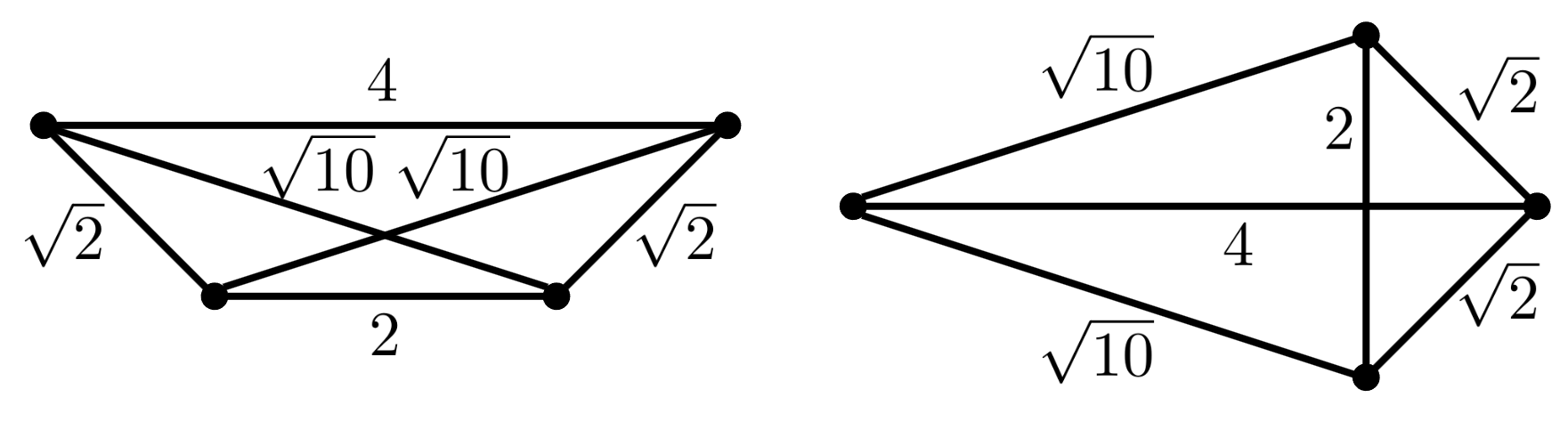}
    \caption{Two weakly homometric point clouds that are not isometric. This particular example comes from \cite[Figure 4]{BK}.}
    \label{fig:genericneeded}
\end{figure}

Rigidity theory deals in questions about the set of frameworks that have the same 
edge lengths as some fixed $(G,P)$.  Intuitively, we regard 
$(G,P)$ as a structure made of fixed length bars connected at their endpoints 
by freely rotating (universal) joints and study its allowed motions.
\begin{definition}\label{def: rigidity}
Let $(G,P)$ and $(G,Q)$ be $d$-dimensional frameworks.  We say that $(G,P)$ and 
$(G,Q)$ are \emph{equivalent} if $f_G(P) = f_G(Q)$, i.e., they have the 
same edge lengths and that $P$ and $Q$ are \emph{congruent} if 
\[
    \|p_i - p_j\|^2 = \|q_i - q_j\|^2 \qquad \text{for all $\{i,j\}\in \binom{[n]}{2}$}
\]
A framework $(G,P)$ is \emph{(locally) rigid} if there is a neighborhood $U\ni P$ 
such that if $Q\in U$ and $(G,P)$ is equivalent to $(G,Q)$, then $P$ is 
congruent to $Q$.  A framework $(G,P)$ is \emph{globally rigid} if \emph{any} framework 
that is equivalent to $(G,P)$ is congruent to it.  A framework  that is not 
rigid is called \emph{flexible}.
\end{definition}

In Figure \ref{fig:rigidex1} we illustrate examples of flexible, rigid, and globally rigid frameworks.

\begin{figure}[htpb]
    \centering
    \includegraphics[width=0.8\textwidth]{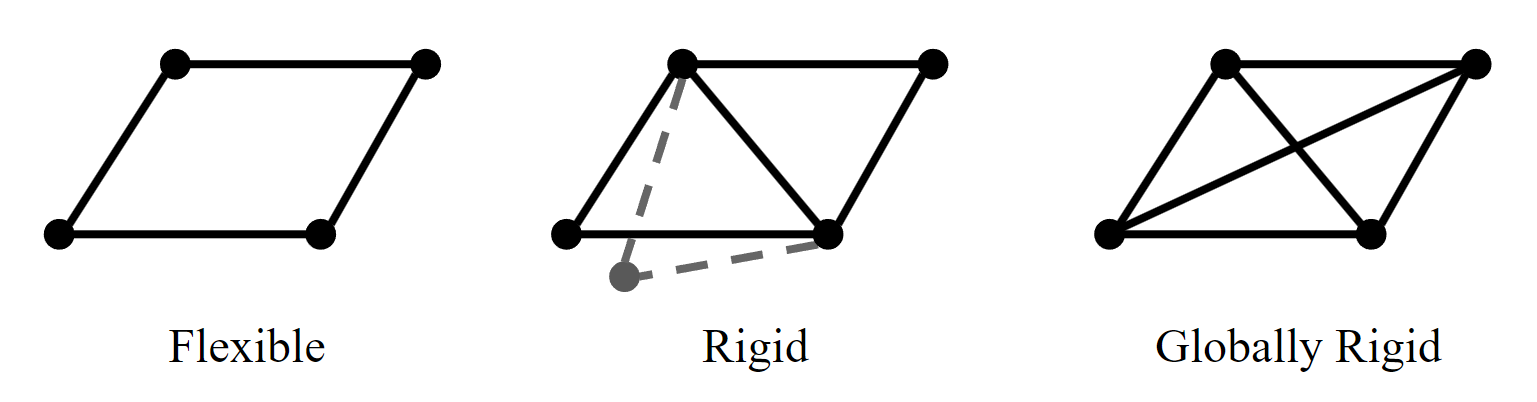}
    \caption{Examples of frameworks in $\RR^2$.  The framework on the left is flexible, since it 
    can deform in a way that expands one diagonal of the quadrilateral and contracts the other while 
    holding the edge lengths fixed.  The framework in the middle is rigid.  If we pin down a triangle to 
    factor out ambient isometries, there are two positions for the remaining unpinned vertex that keep 
    the remaining edge lengths fixed, so any equivalent but non-congruent frameworks must be at some 
    fixed distance from each other.  The framework on the right is complete, so it is globally rigid by definition.  We 
    can also see that adding the missing edge to the framework in the middle resolves the ``flip ambiguity''.}
    \label{fig:rigidex1}
\end{figure}

Frameworks with the same graph can have different rigidity properties.  However, 
for each fixed graph, there is a generic behavior.

\begin{theorem}[{\cite{asimow1978rigidity,GHT}}]\label{thm: generic rigidity}
Rigidity and global rigidity are generic properties.  That is, for each 
fixed dimension $d\in \NN$ and $n\in \NN$, and every graph $G$ on $n$ 
vertices, there is a Zariski open subset $U$ of $d$-dimensional configurations 
so that either every $(G,P)$ with $P$ in $U$ is (globally) rigid or every 
$(G,P)$ with $P\in U$ is not (globally) rigid.
\end{theorem}
The statement for rigidity, which follows from standard differential geometry arguments, is due to 
Asimow and Roth \cite{asimow1978rigidity}.  That global rigidity is a generic property of $G$ 
is a non-trivial result of Gortler--Healy--Thurston \cite{GHT}, building on earlier work 
of Connelly \cite{connelly}.  In light of Theorem \ref{thm: generic rigidity}, we can define 
graphs to be rigid.
\begin{definition}\label{def: graph rigidity}
A graph $G$ is  generically rigid (GLR) in dimension $d$ if every generic $d$-dimensional framework $(G,P)$ 
is  rigid. A graph $G$ is  generically globally rigid (GGR) in dimension $d$ if every generic $d$-dimensional 
framework $(G,P)$ is  globally rigid.
\end{definition}

In Figure \ref{fig:rigidex2} we show an example of a graph $G$ which is generically rigid, but for certain point clouds $P$ can have $(G,P)$ flexible, or $(G,P)$ globally rigid.

\begin{figure}[htbp]
    \centering
    \includegraphics[width=0.8\textwidth]{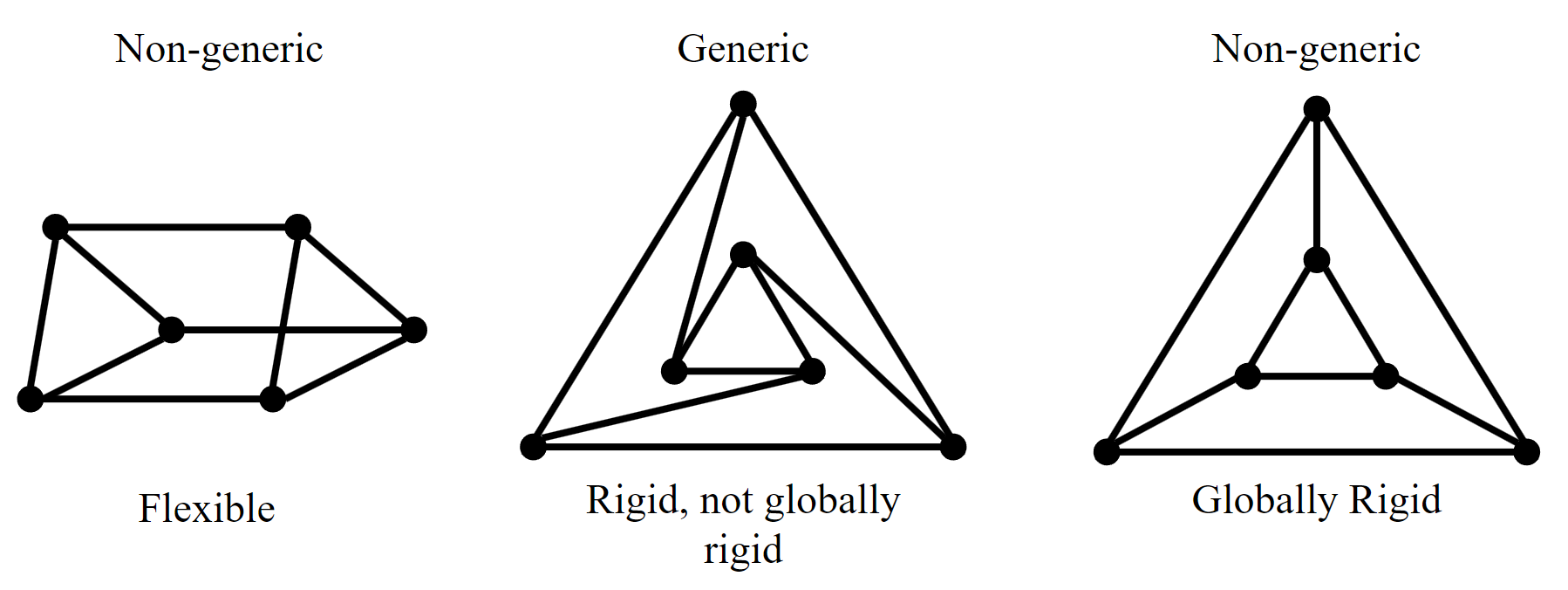}
    \caption{A GLR, but not GGR, graph $G$ which can non-generically be either flexible or globally rigid. On the left we see a non-generic framework $(G,P)$ that is flexible (one triangle can be rotated around the other). In the center, a generic framework that is rigid but not globally rigid (rotate the inner triangle by 120 degrees counterclockwise). On the right a non-generic framework that is globally rigid, see \cite{connelly1996second}.}
    \label{fig:rigidex2}
\end{figure}

In general, there is no simple combinatorial description of GGR or GLR graphs in three or greater dimensions. More precisely, there is no known deterministic polynomial time checkable description of which graphs are GLR or GGR in $d>2$ dimensions. However, whether a graph is generically rigid or generically globally rigid in dimension $d$ can be tested in randomized 
polynomial time, whereas determining whether a particular framework $(G,P)$ is rigid is co-NP hard. Whether or not a graph is GLR or GGR when $d\leq 2$ can be checked in polynomial time, however. A few graph-theoretic definitions are needed.

\begin{definition}
    A graph $G$ is \emph{$k$-connnected} if any graph $G'$ obtained from $G$ by removing $k$ vertices and all edges incident to these vertices is connected. A graph $G$ is \emph{redundantly rigid} if any graph $G'$ obtained from $G$ by removing a single edge is GLR.
\end{definition}

\begin{theorem}[\cite{pollaczek1927gliederung, jackson2005connected}]
\label{thm:dim2rigid}
A graph $G$ is GLR in dimension 2 if and only if $G$ contains a spanning subgraph $G'$ such that
    \begin{enumerate}
        \item $|E(G')| = 2n -3$, and
        \item if $X$ is a subset of $[n]$, the number of edges of $G'$ whose endpoints are both in $X$ is less than or equal to $2|X|-3$.
    \end{enumerate}
    A graph $G$ is GGR in dimension 2 if and only if $G$ is complete or
    \begin{enumerate}
        \item $G$ is 3-connected, and
        \item $G$ is redundantly rigid in dimension 2.
    \end{enumerate}
\end{theorem}

The characterization of graphs that are GLR in dimension 2 was established in \cite{pollaczek1927gliederung} and later rediscovered in \cite{laman1970graphs}. 
The result for GGR graphs was proven decades later still in \cite{jackson2005connected}. For GGR graphs, the phrase  ``$G$ is complete'' is only needed in the second part of the above theorem when $G$ has 3 or fewer vertices, as graphs with this many vertices are never redundantly rigid.

\begin{remark}\label{rem: kinds of genericity}
In the rigidity theory literature, theorems about generic frameworks are usually stated in terms 
of $P$ having coordinates that are algebraically independent over $\QQ$.  As discussed
in \cite{gortler2019generic}, the theorems we use here hold with the genericity 
assumption of Definition \ref{def: generic property}.
\end{remark}

In what 
follows, we will need a strengthening of global rigidity to the situation where the know the 
edge lengths of $(G,P)$ but not $G$.  Here is the relevant combinatorial definition.
\begin{definition}
Let $G = ([n],E)$ and $H = ([n],F)$ be ordered graphs with $m$ edges and 
bijections $g : E\to [m]$ and $f : F\to [m]$ giving the edge orderings.  We say 
that $G$ and $H$ are isomorphic as ordered graphs if there is a permutations $\sigma\in \operatorname{Sym}([n])$ 
and $\tau \in \operatorname{Sym}([m])$ so that, so that 
\[
    \{\sigma(i),\sigma(j)\}\in F \quad \Longleftrightarrow \quad \{i,j\}\in E
    \qquad \text{and} \qquad \tau(f(\{i,j\})) = g(\{\sigma(i),\sigma(j)\}) \quad \text{(all $\{i,j\} \in E$)}
\]
In words, the graph isomorphism $\sigma$ reorders the edges of $G$ according to $\tau^{-1}$.
\end{definition}
We will use the following result about ``unlabeled'' generic global rigidity, which says, 
informally, that the unordered multi-set of edge lengths from a generic, globally 
rigid framework determine the graph, up to isomorphism, and the configuration, up 
to congruence.
\begin{theorem}[{\cite{gortler2019generic}}] \label{thm: gugr}
Suppose $n\geq d+2$ and $d\geq 2$. Let $G$ be an ordered GGR graph in dimension $d$ with $n$ vertices and $(G,P)$ a generic 
$d$-dimensional framework. If  $H$ is any other ordered graph with $n$ vertices and $m$ 
edges, and there is a $\tau\in \operatorname{Sym}([m])$ so that 
\[
    f_H(Q) = \tau(f_G(P))
\]
then there is a $\sigma\in \operatorname{Sym([n])}$ that makes $H$ isomorphic to 
$G$ with the same $\tau$ and, under the vertex relabeling $\sigma$, 
$Q$ is congruent to $P$.
\end{theorem}

\begin{corollary}[\cite{BK}]
    \label{cor:bk}
    Let $P,Q \in \PC$ and suppose $n\geq d+2$ and $d\geq 2$. If $P$ is generic and $P$ and $Q$ are weakly homometric then $P$ and $Q$ are isometric.
\end{corollary}

\begin{proof}
    Take $G$ to be an ordered complete graph, which is GGR, and apply Theorem \ref{thm: gugr}.
\end{proof}

We note that Corollary \ref{cor:bk} was proven by Boutin and Kemper before Theorem \ref{thm: gugr}. The genericity assumption is necessary for Corollary \ref{cor:bk}, as is shown by Figure \ref{fig:genericneeded}. Since Corollary \ref{cor:bk} is a special case of Theorem \ref{thm: gugr} we see the genericity assumption is necessary there as well. We remark here that this result of Boutin and Kemper has been used to motivate studying pairwise distances of point clouds in data science settings, for example in \cite{widdowson2021pointwise}.

\section{Semialgebraic Structure}
\label{sec:semialgebraic}
The goal of this section is to show that given a full barcode $D$, the level sets $(\PHc)^{-1}(D)$ and $(\PHvr)^{-1}(D)$ in $\PC$ both have a semialgebraic structure, closely following arguments from \cite{carriere2021optimizing}. This will establish that the notion of dimension established in Section \ref{sec:dimdef} is well defined for fibers of the persistence map. Moreover, the observations we make along the way will help us determine a generic lower bound on the dimensions of the fibers of these two persistence maps, and will be useful later in establishing connections to rigidity theory.

In greater generality, we want to show that persistence arising from any $\Phi$, an order preserving map parametrized by $\PC$, has semialgebraic level sets. This is more general since it is straightforward to check that $\Phi^{\VR}:\PC \to \mathbb{R}^{|K|}$ is semialgebraic, since each $\Phi^\VR_\sigma$ is semialgebraic. Moreover, the following lemma shows that $\Phi^{\C}$ is also semialgebraic.

\begin{proposition}
The map $\Phi^{\C}$ is semialgebraic.    
\end{proposition}

\begin{proof}
    The coordinate maps of $\Phi_\sigma^{\C}$ are the maps $\rho_\sigma$, which assign to $P$ the enclosing radius of $\sigma(P)$. These maps are semialgebraic by Lemma \ref{lem:semialgenclosing}. Thus $\Phi^{\C}$ is semialgebraic, since it is coordinate-wise semialgebraic.
\end{proof}

The main result of this section is the following:

\begin{theorem}
    \label{thm:fibersemialg}
    Let $K$ be a finite simplicial complex. Suppose that $\Phi:\PC\to\mathbb{R}^{|K|}$ is a semialgebraic order preserving map parametrized by $\PC$. Then given a barcode $B$, the set $(\PH_i\circ\Phi)^{-1}(B)$ is semialgebraic for all $i \geq 0$. It follows that the set $(\PH_i\circ\Phi)^{-1}(D)$ is semialgebraic for any full barcode $D$.
\end{theorem}

Our method of proof is based on those of \cite{carriere2021optimizing}, and will involve the following construction.
\begin{definition}
    Given a set $W$, a \emph{total preorder} $\preceq$ on $W$ is a relation such that
    \begin{enumerate}
        \item For all $x\in W$, $x\preceq x$.
        \item For all $x,y\in W$, either $x\preceq y$, $y \preceq x$, or both. 
        \item If $x\preceq y$ and $y\preceq z$, then $x\preceq z$.
    \end{enumerate}
    We denote by $\preord(W)$ the set of total preorders $\preceq$ of $W$. If $x \preceq y$ but not $y \preceq x$ we write $x \prec y$.
\end{definition}
This definition is relevant to the above theorem because, given any $P\in\PC$, $\Phi$ as stated in the theorem assigns a total preorder to the simplices of $K$ via the rule $\tau \preceq \sigma$ whenever $\Phi_\tau(P) \leq \Phi_\sigma(P)$. As such, we get a map from point clouds into the set of total preorders on $K$. 

\begin{definition}
    Let $K$ be a simplicial complex and $\Phi$ be an order preserving map parametrized by $\PC$. Define the map $\mathcal{T}: \PC \to \preord(K)$ such that $\mathcal{T}(P) = \preceq$ where $\preceq$ is the relation satisfying
    \begin{equation*}
        \tau \preceq \sigma \iff \Phi_\tau(P) \leq \Phi_\sigma(P).
    \end{equation*}
    We let $S_\preceq := \mathcal{T}^{-1}(\preceq)$.
\end{definition}

As we will observe, the persistence map is well-behaved on the subspaces $S_\preceq\subseteq \PC$. The reason for this is that the sets $S_\preceq$ consist of point clouds where simplices appear in the same order, the order determined by $\preceq$. We illustrate this scenario in Figure \ref{fig:partition}.
\begin{figure}[htpb]
    \centering
    \includegraphics[width=0.8\textwidth]{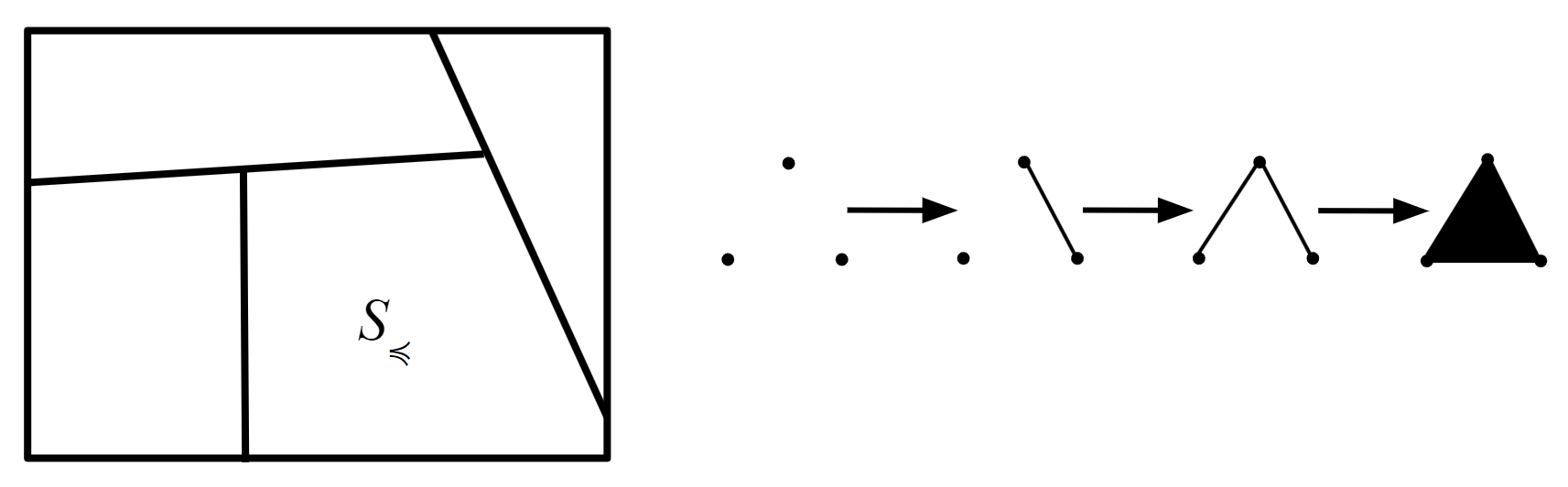}
    \caption{On the left we see the sets $S_\preceq$ partition point cloud space. On the right, a particular filtration assiciated to a set $S_\preceq$. The definition of $S_\preceq$ forces each point cloud $P\in S_\preceq$ to have their simplices appear in the same order in their associated filtration, for example the order in the filtration on the right.}
    \label{fig:partition}
\end{figure}

Insofar as proving Theorem \ref{thm:fibersemialg} is concerned, we care about the sets $S_\preceq$ since we have that
\begin{equation*}
    (\PH_i\circ \Phi)^{-1}(B) = \bigcup_{\preceq \in \preord(K)} (\PH_i\circ \Phi)^{-1}(B)\cap S_\preceq.
\end{equation*}
Since there are only finitely many elements in $\preord(K)$, this union is finite and so it suffices to prove that each set in this union is semialgebraic to prove that the union itself is semialgebraic. We will need often that persistence is well behaved on the sets $S_\preceq$ later, so let us make it a result of its own.

\begin{lemma}
    \label{lem:preorderfiber}
    Let $K$ be a finite simplicial complex. Suppose that $\Phi:\PC\to\mathbb{R}^{|K|}$ is a semialgebraic order preserving map parametrized by $\PC$. Then for any total preorder $\preceq\in\preord(K)$, $S_\preceq$ is semialgebraic. Moreover, $P, Q\in S_\preceq$ have $\PH_i(\Phi(P))=\PH_i(\Phi(Q)) = B$ if and only if, for every bounded endpoint $b$ of $B$ and $\sigma \in K$,
    \begin{equation*}
        \Phi_\sigma(P) = b \iff \Phi_\sigma(Q) = b.
    \end{equation*}
\end{lemma}

\begin{proof}
    The set $S_\preceq$ consists is the set of point clouds $P\in\PC$ such that for all $\sigma,\sigma'\in K$,
    \begin{equation*}
        \Phi_\sigma(P) = \Phi_{\sigma'}(P) \iff \Phi_\sigma(P) - \Phi_{\sigma'}(P) = 0
    \end{equation*}
    if both $\sigma \preceq \sigma'$ and $\sigma'\preceq \sigma$, and
    \begin{equation*}
        \Phi_\sigma(P) > \Phi_{\sigma'}(P) \iff \Phi_\sigma(P) - \Phi_{\sigma'}(P) > 0
    \end{equation*}
    if $\sigma' \prec \sigma$.
    Both the above inequalties and equalities each define a semialgebraic set. Since finitely many such equations define $S_\preceq$, it is the intersection of finitely many semialgebraic sets, and so is semialgebraic. Note that these constraints imply that for any $P,Q\in S_\preceq$, there is a strictly increasing map $\psi: \mathbb{R} \to \mathbb{R}$ such that $\psi \circ \Phi_\sigma(P) = \Phi_\sigma(Q)$ for each $\sigma \in K$. Let us see how the condition of $P$ and $Q$ having the same barcode further constrains $\psi$.

    Fix $P \in S_\preceq$ such that $\PH_i(\Phi(P)) = B$. For any $Q \in S_\preceq$ there exists a strictly increasing map $\psi: \mathbb{R} \to \mathbb{R}$ such that $\psi\circ\Phi_\sigma(P) = \Phi_\sigma(Q)$ for all $\sigma\in K$. If additionally we have $\PH_i(\Phi(Q)) = B$, then by \cite[Lemma 1.5]{leygonie2022fiber}\footnote{While this result technically only applies to $\PH$, by inspecting the proof we see the result also holds for $\PH_i$, for each $i\geq 0$.}, for any value $b$ that is the endpoint of an interval of $B$, we have $\psi(b) = b$. Conversely, fix any strictly increasing map $\psi:\mathbb{R} \to \mathbb{R}$ such that $\psi(b) = b$ for every interval endpoint $b$ in $B$. Given $Q\in\PC$ satisfying $\psi\circ\Phi_\sigma(P) = \Phi_\sigma(Q)$ for all $\sigma \in K$, we observe that $\mathcal{T}(P) = \mathcal{T}(Q)$ and, by \cite[Lemma 1.5]{leygonie2022fiber}, $\PH_i(\Phi(P)) = \PH_i(\Phi(Q)) = B$. In summary, $Q\in S_\preceq$ has $\PH_i(\Phi(Q)) = B$ if and only if there exists a strictly increasing function $\psi$ fixing endpoints of intervals in $B$ such that $\psi\circ\Phi_\sigma(P) = \Phi_\sigma(Q)$ for all $\sigma\in K$. This happens for $P,Q\in S_\preceq$ if and only if $\Phi_\sigma(P) = b \iff \Phi_\sigma(Q) = b$ for all $\sigma\in K$, and all barcode endpoints $b$ of $B$, as desired.
\end{proof}

\begin{proof}[Proof of Theorem \ref{thm:fibersemialg}]
    To show that $(\PH_i\circ \Phi)^{-1}(B)$ is semialgebraic it suffices to show that, for each $\preceq \in \preord(K)$, the set $(\PH_i\circ \Phi)^{-1}(B)\cap S_\preceq$ is semialgebraic. If this set is empty it is automatically semialgebraic. Otherwise, there exists $P \in (\PH_i\circ \Phi)^{-1}(B)\cap S_\preceq$. Let $W\subseteq K$ be the subset of simplices $\sigma$ such that $\Phi_\sigma(P)$ is an endpoint value of $B$. By Lemma \ref{lem:preorderfiber}, $Q\in (\PH_i\circ \Phi)^{-1}(B)\cap S_\preceq$ if and only if $\Phi_\sigma(Q)=\Phi_\sigma(P)$ for all $\sigma\in W$. Define $X\subseteq \PC$ by
    \begin{equation*}
        X:= \{Q \in \PC: \Phi_\sigma(Q) - \Phi_\sigma(P) = 0 \text{ for all } \sigma \in W\}.
    \end{equation*}
    Thus we have that $(\PH_i\circ \Phi)^{-1}(B)\cap S_\preceq = X \cap S_\preceq$. Moreover $X$ is semialgebraic. To see this, note that the graph of the function $\Phi_\sigma(Q) - \Phi_\sigma(P)$ (as a function of $Q$ with $P$ fixed), which is an element of $\PC \times \mathbb{R}$, is semialgebraic. Hence its intersection with $\PC \times \{0\}$ is semialgebraic, and so by projecting we see that the set of $Q$ such that $\Phi_\sigma(Q) - \Phi_\sigma(P) = 0$ is semialgebraic. By varying $\sigma$ over $W$, we see that $X$ is the intersection of finitely many semialgebraic sets, one for each $\sigma \in W$. So $X$ is indeed semialgebraic.

    The set $S_\preceq$ is semialgebraic by Lemma \ref{lem:preorderfiber}, so $(\PH_i\circ \Phi)^{-1}(B)\cap S_\preceq = X \cap S_\preceq$ is the intersection of two semialgebraic sets, and so is semialgebraic. This moreover shows that $(\PH_i\circ \Phi)^{-1}(B)$ is semialgebraic.

    Now all that remains is to show that given a full barcode $D = \{D_i\}_{i=0}^\infty$ that the set
    \begin{equation*}
        (\PH\circ\Phi)^{-1}(D) = \bigcap_{i=0}^{\infty} (\PH_i\circ \Phi)^{-1}(D_i)
    \end{equation*}
    is semialgebraic. Since $K$ is finite, there is an $N$ such that the dimension of each $\sigma\in K$ is less than $N$. It follows that if $D_i$ is nonempty for some $i\geq N$ then $(\PH_i\circ\Phi)^{-1}(D_i)$ is empty- the homology in degree $i$ of subcomplexes of $K$ is always trivial. In this case it follows $(\PH\circ\Phi)^{-1}(D)$ is empty, and hence semialgebraic. Otherwise, $D_i$ is empty for all $i \geq N$. For such $i$ we have that $(\PH_i\circ\Phi)^{-1}(D_i) = \PC$, again because the homology in degree $i$ of subcomplexes of $K$ is always trivial. In this case,
\begin{equation*}
        (\PH\circ\Phi)^{-1}(D) = \bigcap_{i=0}^{\infty} (\PH_i\circ \Phi)^{-1}(D_i) = \bigcap_{i=0}^{N-1} (\PH_i\circ \Phi)^{-1}(D_i)
\end{equation*}
    Hence $(\PH\circ\Phi)^{-1}(D)$ is a finite intersection of semialgebraic sets, and therefore semialgebraic.
\end{proof}

\begin{corollary}
    Let $D = \{D_i\}_{i = 0}^{\infty}$ be a full barcode. Then the spaces $(\PHc)^{-1}(D)$ and $(\PHvr)^{-1}(D)$ are both semialgebraic, as are the spaces $(\PHc_i)^{-1}(D_i)$ and $(\PHvr_i)^{-1}(D_i)$ for each $i \geq 0$.
\end{corollary}

\begin{proof}
    Applying Theorem \ref{thm:fibersemialg}, this follows immediately from the facts that $\Phi^{\C}$ and $\Phi^{\VR}$ are semialgebraic, combined with the fact that $(\PHc_i)^{-1}(D_i) = (\PH_i\circ \Phi^{\C})^{-1}(D_i)$ and $(\PHvr_i)^{-1}(D_i) = (\PH_i\circ \Phi^{\VR})^{-1}(D_i)$.
\end{proof}

Later we will need to understand the interior and boundary of the sets $S_\preceq$, so we record a basic definition and fact regarding these sets here.

\begin{definition}
    We say that a point cloud $P$ is interior to $S_\preceq$ if there is an open ball $B$ in $\PC$ such that $P\in B \subseteq S_\preceq$.
\end{definition}

The following proposition shows that being interior to a set $S_\preceq$ is a generic condition.

\begin{proposition}
\label{prop:genposn}
The set of point clouds that are not interior to any set $S_\preceq$ is semialgebraic and has dimension less than $nd$.
\end{proposition}

\begin{proof}
    Denote the set of point clouds not interior to any $S_\preceq$ by $N$, and let $N_\preceq := N\cap S_\preceq$. The set of point clouds interior to $S_\preceq$ is by definition the interior of $S_\preceq$, when viewed as a subspace of $\PC$, and hence is semialgebrac. Hence $N$ is the complement of a finite union of semialgebraic sets, and so is semialgebraic, implying that each $N_\preceq$ is also semialgebraic. Also, $N_\preceq$ contains no points in the interior of $S_\preceq$ and so is a subset of the boundary of $S_\preceq$, viewed as a subspace of $\mathbb{R}^{nd}$. We will denote this boundary by $\bd(S_\preceq)$. Using Lemma \ref{lem:boundarydim} in the appendix, we have 
    \begin{equation*}
        \dim N = \max_{\preceq \in \preord(K(n))}\dim N_\preceq \leq \max_{\preceq \in \preord(K(n))} \dim \bd(S_\preceq) < nd,
    \end{equation*}
    giving the result.
\end{proof}

Lastly, we will need that we have the following generic relationship between isometric and congruent point clouds on the sets $S_\preceq$.

\begin{lemma}
    \label{lem:isomeansconj}
    Let $\Phi = \Phi^{\C}$ or $\Phi = \Phi^{\VR}$, and suppose that $P \in \PC$ is generic. If $P,Q \in S_\preceq$ are isometric, then $P$ and $Q$ are congruent. 
\end{lemma}

\begin{proof}
    By genericity of $P$, we may suppose that the pairwise distances $d(p_i,p_j)$ of $P$ are all distinct. Now suppose $P,Q \in S_\preceq$. In $S_\preceq$, the simplices of $K(n)$ always appear in the same order. In particular, the same is true for the pairs $\{i,j\} \in K(n)$. For either choice of $\Phi$ in the theorem, $\Phi_{\{i,j\}}(P) = \frac{1}{2}d(p_i,p_j)$. Therefore, $d(p_i,p_j)$ is the $k^\mathrm{th}$ smallest pairwise distance of points in $P$ if and only if the same is true for the distance $d(q_i,q_j)$ in $Q$.

    If $\phi:P \to Q$ is an isometry, then $P$ and $Q$ have the same multiset of edge lengths. Since $d(p_i,p_j) = d(\phi(p_i),\phi(p_j))$, it follows that $d(\phi(p_i),\phi(p_j))$ is the $k^\mathrm{th}$ smallest pairwise distance of points in $Q$ if and only if $d(p_i,p_j)$ is the $k^\mathrm{th}$ smallest pairwise distance of points in $P$. The latter happens if and only if $d(q_i,q_j)$ is the $k^\mathrm{th}$ smallest pairwise distance of points in $Q$. Hence we have:
    \begin{equation*}
        d(p_i,p_j) = d(\phi(p_i), \phi(p_j)) = d(q_i,q_j).
    \end{equation*}
    This is true for all pairs $\{i,j\}$ where $i\neq j$ and so $P$ and $Q$ are congruent.
\end{proof}

\section{An Upper Bound}
\label{sec:upperbound}

We are now ready to prove our first result about the dimension of level sets of the persistence map. We will make use of the following lemma, which describes the space of point clouds with a particular spanning tree.

\begin{lemma}
    Let $T$ be a tree with vertex set $[n]$ and edge set $E$ with positive real weights $w_{ij}$ for each $\{i,j\}\in E$. The space
    \begin{equation*}
        S = \{(p_1,\ldots,p_n)\in \PC:\text{if }\{i,j\} \in E \text{ then }d(p_i,p_j) = w_{ij}\}.
    \end{equation*}
    is semialgebraic and has dimension less than or equal to $nd - n + 1$
\end{lemma}

In fact, with very little more work it can be shown that the dimension of $S$ is equal to $nd - n + 1$, but we do not need this.

\begin{proof}
    The space $S$ is the intersection of a space defined by finitely many polynomial equalities with $\PC$ a semialgebraic set. Hence $S$ is semialgebraic. Moreover, $S$ is a subspace of 
    \begin{equation*}
        M: = \{(p_1,\ldots,p_n)\in \mathbb{R}^{nd}:\text{if }\{i,j\} \in E \text{ then }d(p_i,p_j) = w_{ij}, \text{for some $k$}\}.
    \end{equation*}
    The space $M$ is semialgebraic by similar reasoning, and moreover is homeomorphic to
    \begin{equation*}
    \mathbb{R}^d \times (\mathbb{S}^{d-1})^{n-1},
    \end{equation*}
    To see this, fix a root $R$ of $T$. This choice fixes an orientation on the edges of $T$ away from $R$. Let $E'$ denote the set of directed edges $(i,j)$ of $T$ from $i$ to $j$. For each $(i,j) \in E'$, we have a unit vector $v_{(i,j)}(P)= (p_j-p_i)/\|p_j-p_i\|$ indicating the direction from $p_i$ to $p_j$. Pick an ordering $e_1,\ldots, e_{n-1}$ of the elements of $E'$. Since a point cloud $P$ in $M$ is specified by
    \begin{enumerate}
    \item the location of $p_R$, and
    \item the direction $v_{\{i,j\}}(P)$ for each $\{i,j\} \in E'$,
    \end{enumerate}
    we have a homeomorphism
    \begin{align*}
        M &\to \mathbb{R}^d \times (\mathbb{S}^{d-1})^{n-1}\\
        P &\mapsto (p_R, v_{e_1}(P), \ldots, v_{e_{n-1}}(P))
    \end{align*}
    Hence $M$ is homeomorphic to $\mathbb{R}^d \times (\mathbb{S}^{d-1})^{n-1}$, which as a manifold has dimension $nd - n + 1$. So $M$ has the same dimension since the dimension of a semialgebraic set that is a manifold agrees with its manifold dimension. Since $S \subseteq M$, $\dim S \leq \dim M$.
\end{proof}

Now we prove an upper bound for the dimension of fibers of the persistence map.

\dimup*

\begin{proof}

Fix any point cloud $P = \{p_1,\ldots,p_n\}$ and $D_0=\PHc_0(P)$ (and so $D_0 = \PHvr_0(P)$). Suppose the finite right endpoint values of $D_0$ are $r_1<\ldots< r_m$. By Lemma \ref{lem:minspantree} there exists a tree $T$ with vertex set $[n]$ and edge set $E$ such that if $\{i,j\} \in E$, then $d(p_i,p_j) = 2r_k$ for some $1\leq k \leq m$. Hence $P$ lies in the following subspace of $\mathbb{R}^{nd}$
\begin{equation*}
    \{(q_1,\ldots,q_n)\in \PC:\text{if }\{i,j\} \in E \text{ then }d(q_i,q_j) = 2r_k, \text{for some $k$}\}.
\end{equation*}

The tree $T$ in general depends not only on $D_0$, but also on the geometry of $P$ (and is not necessarily uniquely defined!). However, if we let $\mathcal{S}$ denote the set of trees $T$ with vertex set $[n]$, and $\mathcal{R}(T)$ the set of maps from the edges of a given spanning tree $T$ on $[n]$ to the set $\{2r_1,\ldots, 2r_k\}$ we have for now arbitrary $P$ with $\PHc(P) = D_0$ that
\begin{equation*}
    P \in \bigcup_{T\in \mathcal{S}} \bigcup_{\phi \in \mathcal{R}(T)} \{(q_1,\ldots,q_n)\in \PC: \text{if }\{i,j\} \in E(T) \text{ then }d(q_i,q_j) = \phi(\{i,j\})\}.
\end{equation*}
However, each set
\begin{equation*}
    \{(q_1,\ldots,q_n)\in \PC: \text{if }\{i,j\} \in E(T) \text{ then }d(q_i,q_j) = 2\phi(\{i,j\})\}
\end{equation*}
has dimension $nd - n + 1$ by the previous lemma. Hence $P$ lies in a finite union of semialgebraic sets of dimension at most $nd - n + 1$. Hence $(\PHc)^{-1}(D)$ has dimension at most $nd - n + 1$.

Note both $\mathcal{S}$ and the sets $\mathcal{R}(T)$ are all finite. Thus a point cloud $P$ with \v{C}ech barcode $D$ lies on a union of finitely many $nd - n + 1$ dimensional manifolds. So $\PH^{-1}(D)$ has dimension at most $nd - n + 1$.
\end{proof}

Using the results of \cite{smith2022families} we observe that this upper bound is sharp when $d=2$ and we take \v{C}ech persistence.

\begin{proposition}
    Let $D$ be any full barcode with $D_i$ empty for all $i>0$. In the case where $d = 2$, the space $(\PHc)^{-1}(D)$ has dimension $nd - n + 1 = n+1$.
\end{proposition}

\begin{proof}
Suppose that the $n-1$ bounded right endpoints of $D_0$ are $r_1,\ldots,r_{n-1}$. Let $X \subseteq \PC$ denote the subset of point clouds $P = \{p_1,\ldots,p_n\}\in \PC$ such that
\begin{enumerate}
    \item $d(p_{i+1}, p_i) = 2r_i$, and
    \item the angle between the vectors $p_{j} - p_i$ and $(1,0)$ is less than $\pi / 4$ any $i<j$. 
\end{enumerate}
By methods used in the previous proof we can show that $X$ has dimension $n+1$.

The \v{C}ech filtration of any $P \in X$ is a filtration of closed, locally contractible, and proper subsets of the plane, and hence has trivial persistence in degrees greater than 1 by Alexander duality and the universal coefficient theorem. Meanwhile, by combining Proposition 3.3 and Lemma 3.5 from \cite{smith2022families}, $P$ must also have trivial persistence in degree 1. Hence it remains to show that $P$ has barcode $D_0$ in degree zero.

Let $T$ be the minimal spanning tree of $P$ with edge set $E$. Given $i < j <k$, suppose that $\{i,k\}, \{j,k\} \in E$. Since the vectors $p_{j} - p_i$ and $p_k-p_j$ both make an angle with $(1,0)$ of less than $\pi/4$, the angle between $p_i - p_j$ and $p_k - p_j$ is obtuse. Hence $d(p_i, p_j) < d(p_i, p_k)$. Let $E'$ denote the set $E$, but with $\{i,k\}$ replaced by $\{i,j\}$.
We have that
\begin{equation*}
    \sum_{\{i,j\}\in E'}d(p_i,p_j)  < \sum_{\{i,j\}\in E}d(p_i,p_j),
\end{equation*}
contradicting that $T$ is a minimal spanning tree. Similarly it cannot be the case that $\{i,k\}, \{i,j\} \in E$. Therefore, $E$ consists of the edges $\{i,i+1\}$ for $1\leq i <n$. Lemmas \ref{lem:basicH0} and \ref{lem:minspantree} imply that the \v{C}ech persistence module of $P$ has barcode $D_0$ in degree zero.
\end{proof}

The reason this proof does not apply to Vietoris-Rips filtrations is that there is no guarantee that Vietoris-Rips filtrations of planar point clouds have trivial homology in degree 2 and greater (see for example the proof of Proposition 5.3 in \cite{chambers2010vietoris}). If it were shown that an analogous construction of $X$ for $d$ greater than 2 consists of point clouds with trivial persistence in every degree greater than zero for \v{C}ech and Vietoris-Rips filtrations (we suspect this to be the case), the above proof could be extended to the case $d>2$, and to the Vietoris-Rips setting.

\section{A Lower Bound}

\label{sec:lowerbound}

Now we see how the semialgebraic structure of level sets of the persistence map establishes lower bounds on the generic dimension of fibers of the persistence map for the \v{C}ech and Vietoris-Rips filtrations.

\dimlow*

We will prove this by showing that the same result holds on each $S_\preceq$.

\begin{lemma}
\label{lem:lowerbound}
Let $\PH$ denote either $\PHc$ or $\PHvr$. Given $\preceq \in \preord(K(n))$, there is a semialgebraic subset $A_\preceq\subseteq S_\preceq$ of dimension less than $nd$ with the following property. For each $P \in S_\preceq-A_\preceq$, if there are exactly $k$ distinct bounded interval endpoints in barcodes in $D = \PH(P)$, then $\dim\PH^{-1}(D)\geq nd - k + 1$. 
\end{lemma}

\begin{proof}
If $\PH = \PHc$ let $\Phi = \Phi^{\C}$. Otherwise let $\Phi = \Phi^{\VR}$. We will make crucial use of the fact that for any $P\in \PC$, $\Phi_\sigma(P) = 0$ if and only if $\sigma$ is a singleton. Since for any $P, Q \in S_\preceq$, there exists a strictly increasing $\psi:\mathbb{R} \to \mathbb{R}$ such that $\psi \circ \Phi_\sigma(P) = \Phi_\sigma(Q)$ for all $\sigma \in K$, by \cite[Lemma 1.5]{leygonie2022fiber} $\PH(P)$ and $\PH(Q)$ have the same number of distinct bounded endpoints $k$. Hence on $S_\preceq$, $k$ is independent of $P$.

By Lemma \ref{lem:preorderfiber}, $S_\preceq$ is semialgebraic. Pick any $P_0\in S_\preceq$ and let $D = \PH(P_0)$. Let $0 = b_0 < \ldots < b_{k-1}$ be the distinct bounded endpoint values that appear in intervals of $D$.
Since the homology of either the \v{C}ech or Vietoris-Rips filtration changes at each $b_i$, the simplicial filtration itself must change at each $b_i$, so there must be at least one simplex $\sigma_i$ with $\Phi_{\sigma_i}(P_0) = b_i$ for each $i$. In particular, the simplex $\sigma_0$ must be a singleton. Define
\begin{align*}
    f: S_\preceq &\to \mathbb{R}^{k-1}\\
    P &\mapsto (\Phi_{\sigma_1}(P), \ldots, \Phi_{\sigma_{k-1}}(P)).
\end{align*}
Since the coordinate maps of $\Phi$ are semialgebraic functions, so is $f$, as $S_\preceq$ is semialgebraic. Given any $P, Q \in S_\preceq$, there exists a strictly increasing map $\psi:\mathbb{R} \to \mathbb{R}$ such that $\psi \circ \Phi_\sigma(P) = \Phi_\sigma(P_0)$ for all $\sigma \in K$. 
Hence, by Lemma \ref{lem:preorderfiber}, $\PH(P)$ = $\PH(Q)$ if and only if
\begin{equation*}
    \Phi_\sigma(P) = \psi^{-1} b_i \iff \Phi_\sigma(Q) = \psi^{-1} b_i.
\end{equation*}

From the structure of $S_\preceq$ we deduce that this happens if and only if $\Phi_{\sigma_i}(P) = \Phi_{\sigma_i}(Q)$ for all $i$. In particular we always have that $\Phi_{\sigma_0}(P) = \Phi_{\sigma_0}(Q) = 0$, since $\sigma_0$ is a singleton. Hence $\PH(P)$ = $\PH(Q)$ if and only if $f(P) = f(Q)$. Since we assumed $P,Q \subseteq S_\preceq$, if $D = \PH(P)$ and $a = f(P)$, then $f^{-1}(a) \subseteq \PH^{-1}(D)$. This is an inclusion and not an equality since $\PH^{-1}(D)$ is a subset of $\PC$ while $f^{-1}(a)$ is only a subset of $S_\preceq$.

Let $S_f(l) := \{a \in \mathbb{R}^{k-1}:\dim f^{-1}(a) = l\}$. For $l < nd - k + 1$ we have by Lemma \ref{lem:fiberdim} in the appendix that $S_f(l)$ and $f^{-1}(S_f(l))$ are semialgebraic and
\begin{align*}
    \dim f^{-1}(S_f(l)) = \dim S_f(l) + l \leq k-1 + l < (k - 1) + (nd - k +1) = nd.
\end{align*}
Thus we define
\begin{equation*}
    A_\preceq = \bigcup_{l < nd - k + 1} f^{-1}(S_f(l)).
\end{equation*}
This set is semialgebraic and has dimension less than $nd$, being a finite union of semialgebraic sets of dimension less than $nd$. Moreover, for $P \in S_\preceq - A_\preceq$, 
\begin{equation*}
     \dim \PH^{-1}(\PH(P))  \geq \dim f^{-1}(f(P)) \geq nd -k + 1.
\end{equation*}
This proves the lemma.
\end{proof}

With this lemma we can prove the theorem easily.

\begin{proof}[Proof of Theorem \ref{thm:lowerbound}]
    For each $\preceq \in \preord(K(n))$, the result follows immediately from taking $A_\preceq$ as in the above lemma and letting
    \begin{equation*}
        A  = \bigcup_{\preceq \in \preord(K(n))} A_\preceq.
    \end{equation*}
    Then every point cloud in $\PC - A$ has the desired property and $\dim A < nd$.
\end{proof}

\section{Vietoris-Rips Persistence Meets Rigidity Theory}
\label{sec:vrrigid}
In this section we only consider Vietoris-Rips persistence and write $\PH(P) := \PHvr(P)$ as well as $\Phi:= \Phi^{\VR}$. Here we will look for point clouds that the persistence map describes in a sense as well as possible.

It is well known that $\PH$ is isometry invariant (see \cite{oudot2020inverse} for example). Therefore, if we have a point cloud $P$ and want that $D=\PH(P)$ differentiates $P$ from other point clouds $Q$, a best case scenario would be that any other point cloud $Q$ with $\PH(Q) = D$ is in fact isometric to $P$. Barring this, a second best case scenario would be that this is true for all $Q$ in a neighborhood of $P$. So we have the following definition.

\begin{definition}
    Let $P\in\PC$ satisfy $\PH(P) = D$. We say that $P$ is \emph{identifiable up to isometry (under Vietoris-Rips persistence)} if for all $Q\in \PH^{-1}(D)$, $P$ and $Q$ are isometric.

    We say $P$ is \emph{locally identifiable up to isometry (under Vietoris-Rips persistence)} if there exists a neighborhood $U$ of $P$ in $\PC$ such that every point cloud $Q \in U\cap\PH^{-1}(D)$ is isometric to $P$. 
\end{definition}

We illustrate our notion of local identifiability in Figure \ref{fig:localid}.

\begin{figure}
    \centering
    \includegraphics[width=0.45\textwidth]{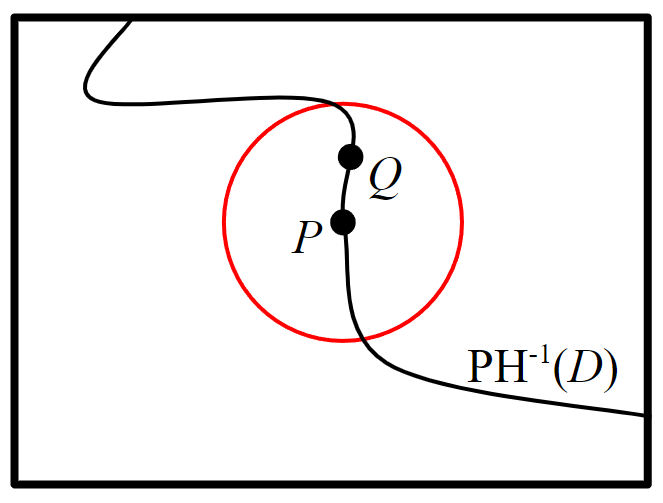}
    \caption{A point cloud $P$ in the fiber of the full barcode $D$, with a neighborhood highlighted. The point cloud $P$ is locally identifiable if every $Q$ inside the fiber of $D$ and the neighborhood is isometric to $P$.}
    \label{fig:localid}
\end{figure}

Now we begin to characterize the point clouds that are identifiable up to isometry. It will be convenient to first study point clouds that are locally identifiable up to isometry. The key idea is that Vietoris-Rips persistence only records metric information about pairwise distances between points in a point cloud $P$. We need a notion of the pairs of points whose distances are indeed recorded by the Vietoris-Rips filtration.

\begin{definition}
    Let $P= (p_1,\ldots,p_n)\in \PC$ with $D = \PH(P)$. For $i\neq j$, we say that $\{i,j\}$ is a \emph{critical edge} if $D$ has $\frac{1}{2}d(p_i,p_j)$ as a bounded endpoint. We denote the set of critical edges of $P$ by $\crit(P)$.
\end{definition}

Near a generic point cloud $P$, another point cloud $Q$ has the same full barcode if and only if they have the same lengths of edges determined by $\crit(P)$. To show this we will make use of the following lemma.

\begin{lemma}
\label{lem:critsame}
    If $P,Q\in S_\preceq$ then $\crit(P) = \crit(Q)$.
\end{lemma}

\begin{proof}
    Since $P,Q\in S_\preceq$ we may pick a strictly increasing map $\psi:\mathbb{R} \to \mathbb{R}$ such that $\psi \circ \Phi_\sigma(P) = \Phi_\sigma(Q)$ for each simplex $\sigma \in K$.  If $\{i,j\} \in \crit(P)$, then $b = \frac{1}{2}d(p_i,p_j)$ is a bounded endpoint of $D$. By \cite[Lemma 1.5]{leygonie2022fiber}, $\psi(b)$ is a bounded endpoint of $\PH(Q)$. Moreover $\Phi_{\{i,j\}}(Q) = \psi \circ \Phi_{\{i,j\}}(P) = \psi(b)$. Hence $\{i,j\} \in \crit(Q)$. So $\crit(P) \subseteq \crit(Q)$. Similarly $\crit(Q) \subseteq \crit(P)$.
\end{proof}

\begin{proposition}
    \label{prop:critdists}
    Let $P=(p_1,\ldots,p_n)\in \PC$ be interior to $S_\preceq$. There is a neighborhood $U$ of $P$ such that for all $Q=(q_1,\ldots,q_n) \in U$, $\PH(P) = \PH(Q)$ if and only if for all $\{i,j\} \in \crit(P)$, $d(p_i,p_j) = d(q_i,q_j)$.
\end{proposition}

\begin{proof}
Set $U$ to be the interior of $S_\preceq$. It follows that $U$ contains $P$. Let $D = \PH(P)$.

($\implies$) Suppose $Q = (q_1,\ldots,q_n)\in U$ satisfies $\PH(Q) = D$. If $\{i,j\}\in \crit(P)$, then $\frac{1}{2} d(p_i,p_j) = b$, for some bounded endpoint $b$ of $D$.

By Lemma \ref{lem:preorderfiber} we have
\begin{equation*}
    d(p_i,p_j) = 2\Phi_{\{ i,j \}}(P) = 2\Phi_{\{ i,j \}}(Q) = d(q_i,q_j).
\end{equation*}

($\impliedby$) Suppose $Q=(q_1,\ldots,q_n)\in U$ has that for all $\{i,j\}\in \crit(P)$, $d(p_i,p_j) = d(q_i,q_j)$. Let $b$ be any bounded endpoint of $D$, and $\sigma$ be a simplex such that $\Phi_\sigma(Q) = b$. The fact that $P,Q \in S_\preceq$ implies that the maxima defining $\Phi_\sigma(P)$ and $\Phi_\sigma(Q)$ are both attained by the same pair of indices $\{i_0,j_0\}$. Since we have $b = \Phi_\sigma(Q) = \frac{1}{2}d(q_{i_0},q_{j_0})$, we observe $\{i_0,j_0\}\in \crit(Q)$. By Lemma \ref{lem:critsame}, we therefore have $\{i_0,j_0\}\in \crit(P)$. Thus,
\begin{equation*}
    \Phi_\sigma(P) = \frac{1}{2}d(p_{i_0},p_{j_0}) = \frac{1}{2}d(q_{i_0},q_{j_0}) = \Phi_\sigma(Q) = b.
\end{equation*}
The same argument (except that we do not need to invoke Lemma \ref{lem:critsame}) shows that if $b$ is a bounded endpoint of $D$ and $\sigma$ is a simplex such that $\Phi_\sigma(P) = b$, then $\Phi_\sigma(Q) = b$.

Therefore, for all bounded endpoints $b$ of $D$ and $\sigma \in K$, $\Phi_\sigma(P)$, we have that $\Phi_\sigma(P) = b$ if and only if $\Phi_\sigma(Q) = b$. Lemma \ref{lem:preorderfiber} implies that $\PH(P) = \PH(Q)$.
\end{proof}

We now define the critical graph of a point cloud in the obvious way.

\begin{definition}
    Given $P \in \PC$, we define the \emph{critical graph} of $P$ to be the graph $G_P$ with $N(G_P) = [n]$ and $E(G_P) = \crit(P)$.
\end{definition}

We immediately have the following result from Proposition $\ref{prop:critdists}$.

\begin{restatable}{theorem}{vrglr}
\label{thm:vrrigid}
    Let $P\in\PC$ be a generic point cloud. $P$ is locally identifiable up to isometry under Vietoris-Rips persistence if and only if $(G_P, P)$ is rigid.
\end{restatable}

\begin{proof}
    Let $P = (p_1,\ldots,p_n)$. By genericity of $P$ we may assume $P$ is to some $S_\preceq$, using Proposition \ref{prop:genposn}. 
    
    ($\impliedby$) Suppose $(G_P,P)$ is rigid. Let $U$ be the intersection of neighborhoods of $P$ given by Proposition \ref{prop:critdists} and by the rigidity of $(G_P,P)$. Let $Q = (q_1,\ldots,q_n)$ be another element of $U$ satisfying $\PH(P) = \PH(Q)$. Proposition \ref{prop:critdists} implies that $Q$ satisfies $d(p_i,p_j) = d(q_i,q_j)$ for all $\{i,j\}\in \crit(P)$. The rigidity of $(G_P,P)$ implies that $Q$ is congruent to $P$, and hence isometric to $P$. So $P$ is locally identifiable up to isometry. 
    
    ($\implies$) Suppose $P$ is locally identifiable up to isometry. Let $U$ be the intersection of the neighborhoods of $P$ given by Proposition \ref{prop:critdists} and the local identifiablility of $P$ up to isometry. Let $Q = (q_1,\ldots,q_n)$ be another element of $U$, such that $d(p_i,p_j) = d(q_i,q_j)$ for all $\{i,j\} \in \crit(P)$. Proposition \ref{prop:critdists} implies that $\PH(P) = \PH(Q)$. Local identifiability of $P$ implies that $P$ and $Q$ are isometric. By genericity of $P$ and Lemma \ref{lem:isomeansconj}, $P$ and $Q$ are congruent.
\end{proof}

Below we present a couple basic examples of the theorem in action.

\begin{example}
    Let $P = \{p_1,p_2,p_3\}$ be a point cloud of three points. Without loss of generality, assume that $d(p_1,p_2)$ is the largest of the pairwise distances. By assuming genericity of $P$, we may take $d(p_1,p_2)$ to be the strictly largest of the pairwise distances. It is then straightforward to check that $\{2,3\}$ and $\{1,3\}$ are critical edges of $P$, corresponding to the two bounded right endpoints of $\PH_0(P)$. However, $\{1,2\}$ is not a critical edge. This is because the inclusion of this edge cannot change degree zero homology as the filtration of $K(3)$ is already connected for values below $d(p_1,p_2)$. The only other possibility is that $\{1,2\}$ changes homology by creating a cycle. However this also cannot happen as $\{1,2\}$ appears at the same time in the Vietoris-Rips filtration as the 2-simplex $\{1,2,3\}$. Hence, $G_P$ is the graph on $[n]$ with edges $\{1,3\}$ and $\{2,3\}$, and this graph is not rigid. Hence we deduce that generic point clouds with three points are not locally identifiable.
\end{example}

\begin{example}
    \label{ex:vrexample}
    Let $P = (p_1,p_2,p_3,p_4)$ be a planar point cloud given by the vertices of a square in clockwise order, perturbed slightly so that $P$ is generic. For convenience we write $d_1:=d(p_1,p_2)$, $d_2:= d(p_2,p_3)$, $d_3:=d(p_3,p_4)$, $d_4:=d(p_4,p_1)$ and $d_1':=d(p_1,p_3)$, $d_2':=d(p_2,p_4)$. We show a possible filtration of such a point cloud in Figure \ref{fig:vrexample}.

    \begin{figure}
    \centering
    \includegraphics[width=0.8\textwidth]{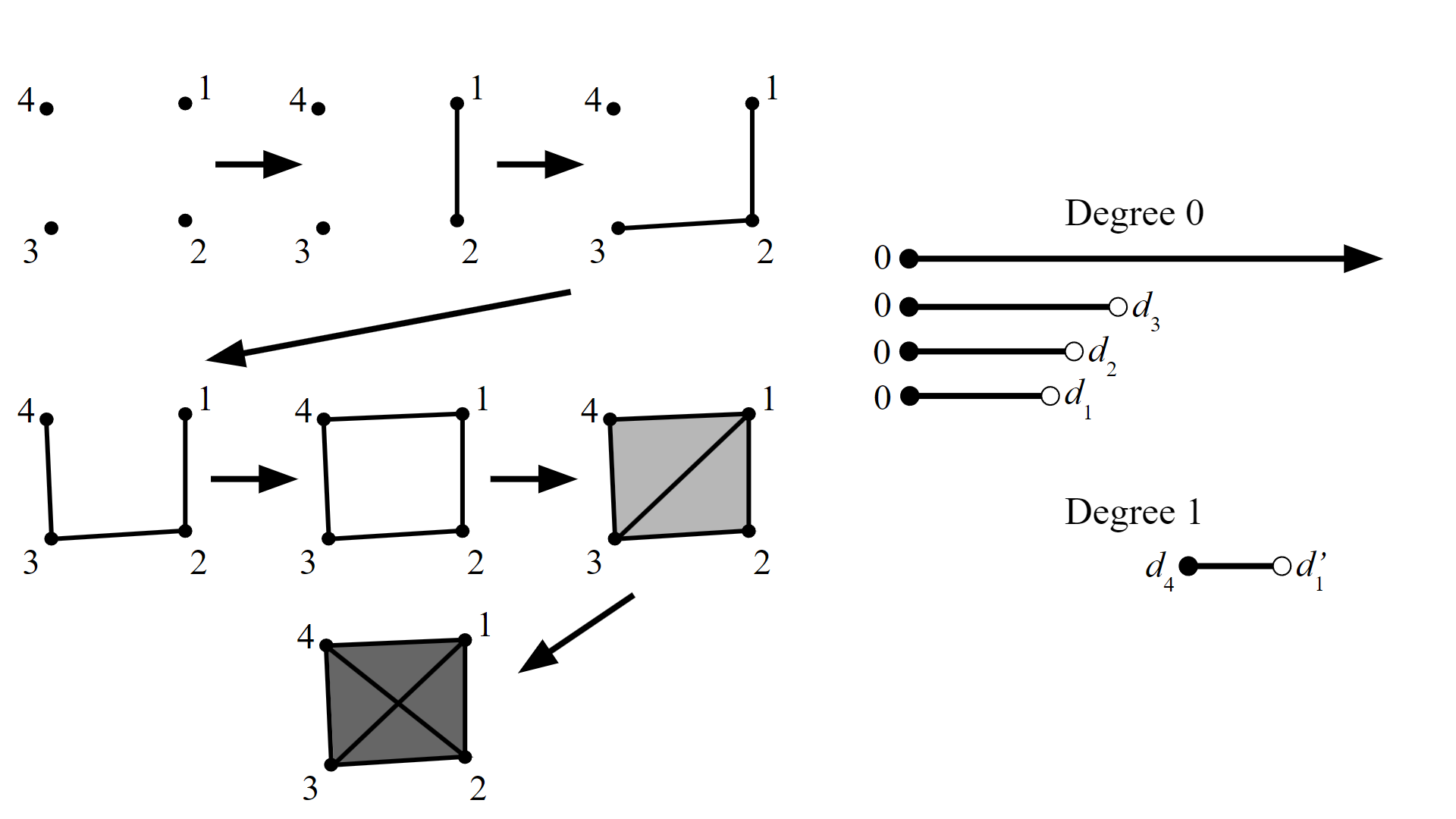}
    \caption{Left: A possible filtration given by a point cloud $P$ as in Example \ref{ex:vrexample}. Right: The full barcode arising from this filtration.}
    \label{fig:vrexample}
\end{figure}
    
    Assuming the perturbation is sufficiently small we have that $d_1$, $d_2$, $d_3$, and $d_4$ are less than both $d_1'$ and $d_2'$. As such we have that the smallest three of $d_1$ through $d_4$ determine the three bounded right endpoints of $\PH_0(P)$, whereas the inclusion of the edge corresponding to the largest of $d_1$ through $d_4$ creates a cycle, and hence determines a left endpoint in $\PH_1(P)$. As such, the edges $\{1,2\}$, $\{2,3\}$, $\{3,4\}$, and $\{1,4\}$ are all critical. Without loss of generality, suppose that $d_1' \leq d_2'$. Then $\{1,2,3\}$ and $\{1,3,4\}$ both have filtration value $d_1'$ since the simplex $\{1,3\}$ appears at this filtration value, and all other subsimplices of these 2-simplices appear earlier. In particular, the cycle given by the first four edges becomes a boundary at filtration value $d_1'$. We deduce that $\{1,3\}$ is a critical edge. The point cloud $P$ equipped with the graph $G$ on $[n]$ with edges $\{1,2\}$, $\{2,3\}$, $\{3,4\}$, $\{1,4\}$, and $\{1,3\}$ is known to be rigid. We deduce that $P$ is locally identifiable. 
\end{example}

As a consequence of Theorem \ref{thm:vrrigid}, when $d=2$ we have the following characterization of generic point clouds that are locally identifiable up to isometry. 

\begin{corollary}
    \label{cor:matrank}
    Any generic $P \in \PCtwo$ is locally identifiable up to isometry if and only if $G_P$ contains a spanning subgraph $G$ satisfying
    \begin{enumerate}
        \item $|E(G)| = 2n -3$, and
        \item if $X$ is a subset of $[n]$, the number of edges of $G$ whose endpoints are both in $X$ is less than or equal to $2|X|-3$.
    \end{enumerate}
\end{corollary}

\begin{proof}
This follows immediately from Theorem \ref{thm:vrrigid} the first part of Theorem \ref{thm:dim2rigid}.
\end{proof}

Using Theorem \ref{thm: gugr} we can also find a criterion for identifiability up to isometry.

\vrggr*

\begin{proof}
Given any full barcode $D$, let $E(D)$ denote the set of nonzero, non-infinite endpoints appearing in $D$. Note that $E(D)$ is a set, and so has no repeated entries. We have a map
\begin{align*}
    \phi_P:\crit(P) &\to E(\PH(P))\\
    \{i,j\}&\mapsto \frac{1}{2}d(p_i,p_j).
\end{align*}
The structure of the barcode $D$ implies that this map is surjective for any $P$.
The condition that for all pairs of critical edges $\{i,j\} \neq \{k,l\}$, $d(p_i,p_j)\neq d(p_{k},p_{l})$ is generically satisfied. By imposing this condition on $P$ we see that $\phi_P$ is also injective, and hence bijective.

Now fix $D = \PH(P)$, and suppose $Q \in \PH^{-1}(D)$. As a result $\phi_{Q}$ is surjective on $E(D)$. Therefore we can choose a subgraph $G'_{Q} \subseteq G_{Q}$ with the same multiset of edge lengths as $G_P$. It follows from Theorem \ref{thm: gugr} that $Q$ and $P$ are isometric. Therefore $P$ is identifiable up to isometry.
\end{proof}

As a consequence of Theorem \ref{thm:globrig} we have the a generic criterion for identifiability up to isometry when $d = 2$.

\begin{corollary}
    Let $n \geq 4$ and $P\in \PCtwo$ be a generic point cloud. If
    \begin{enumerate}
        \item $G_P$ is 3-connected, and
        \item $G_P$ is redundantly rigid in dimension 2,
    \end{enumerate}
    then $P$ is identifiable up to isometry.
\end{corollary}

\begin{proof}
    This follows immediately from Theorem \ref{thm:vrrigid} and the second part of Theorem \ref{thm:dim2rigid}.
\end{proof}

\section{Circumsphere Rigidity Theory}
\label{sec:circrigid}
To study identifiability for persistent homology of the \v{C}ech filtration, 
we need to replace rigidity theory of Euclidean frameworks with an 
analogue involving circumspheres.  The setup we develop follows steps 
that are mostly standard in the rigidity and geometric constraint 
literature, but this specific variant is, to our knowledge, new. 

\begin{definition}\label{def: circumspehre framework}
Let $d\in \NN$ be a fixed dimension and $H = ([n],E)$ a hypergraph 
in which each hyperedge $\sigma\in E$ has between $2$ and $d+1$ endpoints 
and $P\in \PC$.
We call the pair $(H,P)$ a \emph{circumsphere framework}.

We say that two circumsphere frameworks $(H,P)$ and $(H,Q)$ are 
\emph{equivalent} if 
\[
    \det \Delta_{\sigma(P)}\det\Lambda_{\sigma(Q)} = 
    \det \Delta_{\sigma(Q)} \det \Lambda_{\sigma(P)}
    \qquad 
    \text{(for all hyperedges $\sigma\in E$)}
\]

A circumsphere framework $(H,P)$ is \emph{(locally) rigid} if there is a 
neighborhood $U\ni P$ such that, if $Q\in U$ and 
$(H,Q)$ is equivalent to $(H,P)$, then $Q$ is congruent to $P$.

A circumsphere framework $(H,P)$ is \emph{generic} if $P$ 
is generic.
\end{definition}

As a first step, we verify that circumsphere (local) rigidity is a generic property, 
analogously to the results of \cite{asimow1978rigidity}.  The proof is somewhat 
standard, though we need to take a bit of care because the measurement map is a 
rational function.
\begin{lemma}\label{lem: circumsphere generic}
Circumsphere rigidity is a generic property.  That is, for every dimension $d\in \NN$
and $n\in \NN$, there is a Zariski open, dense subset $V$ of 
$\PC$ such that for any hypergraph $H$ with $n$ vertices and
hyperedges consisting of between $2$ and $d+1$ endpoints, either every framework $(H,P)$ with $P\in V$ is circumsphere 
rigid or no $(H,P)$ with $P\in V$ is circumsphere rigid.
\end{lemma}

\begin{proof}
Let $d$ and $n$ be given as in the statement.  We first observe that circumsphere rigidity is a non-trivial property 
in the sense that there are rigid and non-rigid frameworks.  This follows because the minimal circumradius of a $2$-point 
configuration $\{p_i,p_j\}$ is simply $\frac{1}{2}d(p_i,p_j)$.  In particular, if $P$ are $Q$ are congruent in the 
circumsphere sense, they are congruent in the conventional sense as well.  We now see that if $H$ has no 
edges, then any circumsphere framework $(H,P)$ must not be rigid and that if $H$ contains every hyperedge with 
two endpoints it must be rigid (because circumradii are invariant under affine isometry).

We first consider the case $n \geq d+1$, wherein the heart of our argument lies. The case $n < d+1$ is then intuitive, but rather technical. Given a hypergraph $H = ([n],E)$ with $n\ge d + 1$ vertices
and $m$ hyper-edges with between $2$ and 
$d+1$ endpoints.  Define a set $U_0$ of configurations $P$ of $n$ points so that, 
for all $\sigma\in E$, the sub-configuration $\sigma(P)$ is affinely independent.  Because each 
hyperedge $\sigma$ has at most $d+1$ endpoints, $U_0$ is a dense Zariski open subset 
of configuration space.  For all $P$ in $U_0$, the polynomial constraint system
\[
    \det \Delta_{\sigma(P)}\det\Lambda_{\sigma(Q)} = 
    \det \Delta_{\sigma(Q)} \det \Lambda_{\sigma(P)}
    \qquad 
    \text{(for all hyperedges $\sigma\in E$)}
\]
with $Q$ as the variable has $P$ as a solution and no trivial equations of the form 
$0 = 0$.  Set $U$ to be subset of $U_0$ on which the Jacobian of this geometric 
constraint system has its maximum rank, which we denote as $r$.  Becuase $U_0$ is 
dense in an (irreducible) affine space, it is also irreducible.  Hence $U$ is 
a dense, Zariski open subset of $U_0$ (as the rank being less than the maximum is 
defined by determinants).

The constant rank theorem implies that, for $P\in U$, there is a 
neighborhood $V$ of $P$ so that 
\[
    C = V\cap \{ Q : \det \Delta_{\sigma(P)}\det\Lambda_{\sigma(Q)} = 
    \det \Delta_{\sigma(Q)} \det \Lambda_{\sigma(P)} \}
\]
is a smooth manifold of dimension $dn - r$.  The manifold $C$ 
consists of the $Q$ in $V$ so that $(H,Q)$ is equivalent to $(H,P)$.

If $r = nd - \binom{d+1}{2}$, then 
$C$ is of dimension $\binom{d+1}{2}$ and otherwise it has higher dimension. These two cases correspond to rigid and flexible, since there is a 
$\binom{d+1}{2}$-dimensional space of images of $P$ under affine isometries
in any neighborhood of $P$, and all of these are in $C$.  
If $C$ has dimension $\binom{d+1}{2}$, then 
$V$ is the neighborhood in the definition of rigidity; if the dimension is 
larger, then any neighborhood of $P$ will contain an equivalent but 
non-congurent $Q$ in $C$.

Since $U$ is a dense, Zariski open set, we have shown that, for this 
$H$, circumsphere rigidity is a generic property.  Since there 
are finitely many $H$, we can intersect the $U$ for each of 
them, and the lemma follows.

Now we handle the case $n < d+1$. For $n \leq d+1$ let $\aff_{n,d}\subseteq \PC$ denote the subspace of affinely independent point clouds. We note that $\PC - \aff_{n,d}$ is semialgebraic and has dimension less than $nd$. For every pair $n,d$ such that $n \leq d+1$ we construct continuous semialgebraic maps $\Psi_{n,d}: \aff_{n,d} \to \aff_{n, n-1}$ and $\psi_{n,d}: \aff_{n,d} \to \mathbb{R}^{n-1}$ by induction. For the base case, we set $\Psi_{n,d} = \psi_{n,d} = 0$. Let $\iota_d: \mathbb{R}^{d-1} \to \mathbb{R}^{d}$ be the inclusion map sending $(x_1,\ldots, x_{d-1})$ to $(x_1,\ldots, x_{d},0)$. For simplicity we write $\iota = \iota_d$ with the value $d$ to be understood from context. We set $\iota(p_1,\ldots,p_n) = (\iota p_1,\ldots,\iota p_n)$. Given any $P = (p_1,\ldots,p_n)\in \aff_{n,d}$, set $P_- = (p_0,\ldots,p_{n-1})$. Letting $d_i = d(p_n,p_i)$, we note that by affine independence of $P$, there is a unique $p\in \mathbb{R}^d$ such that the distance of $p$ from the $i^\mathrm{th}$ entry of $\iota\Psi_{n-1,d}(P_-)$ is $d_i$ for each $i < n$, and such that the $(n-1)^{\mathrm{th}}$ coordinate of $p$, denoted $\pi_{n-1}(p)$ is positive. Therefore define $\psi_{n,d}(P) = p$. It follows that $\psi_{n,d}$ is continuous, provided we assume inductively that $\Psi_{n-1,d}$ is continuous. The graph of $\psi_{n,d}$ is the set
\begin{equation*}
    \psi_{n,d} = \{(P,p) \in \aff_{n,d} \times \mathbb{R}^d: \pi_{n-1}(p) > 0\, \wedge \, d(p,(\iota\Psi_{n-1,d}(P_-))_i)^2 = d_i^2 \text{ for } 1\leq i < n\},
\end{equation*}
As such, $\psi_{n,d}$ is semialgebraic if we assume $\Psi_{n-1,d}$ is semialgebraic, by quantifier elimination. Now we define $\Psi_{n,d}(P) = (\Psi_{n-1,d}, \psi_{n,d})$, which is therefore also continuous and semialgebraic, again inductively assuming $\Psi_{n-1,d}$ is continuous and semialgebraic. Moreover, by construction, $\Psi_{n,d}(P)$ is congruent to $P$. In fact, $\Psi_{n,d}$ sends each congruence class to a single point cloud.

Let $Z$ denote the image of $\Psi_{n,d}$, which is independent of the choice of $n >d$. There is a homeomorphism:
\begin{align*}
    h: Z \times O(d) \times \mathbb{R}^d &\to \aff_{n,d}\\
        ((p_1,\ldots,p_n),M,v)&\mapsto(Mp_1 + v,\ldots,Mp_n + v).
\end{align*}
Here, the $Z$ coordinate of $h^{-1}$ is precisely the map $\Psi_{n,n-1}$. Now fix some hypergraph $H$ as given in the lemma. By the proof of the case $n\geq d+1$, there exists a dense, Zariski open subset $U\subseteq\aff_{n,n-1}$ in which every $P$ has that, say, $(H,P)$ circumsphere rigid, with the case of every $P$ being not circumsphere rigid being handled similarly to what follows. We write $U' = \Psi_{n,n-1}(U)$. Since the complement of $U$ in $\aff_{n,n-1}$ is algebraic, the complement of $U'$ in $\Psi_{n,n-1}$ in $Z$ is semialgebraic. Moreover, up to a homeomorphism, $\Psi_{n,n-1}$ is a projection map, and projections of dense sets are dense, so $U'$ is dense. Since being circumsphere rigid is a property that is constant along congruence classes in $\mathbb{R}^{n-1}$, every $P \in U'$ is circumsphere rigid.

Now set $U'' = \Psi_{n,d}^{-1}(U')\subseteq \aff_{n,d}$. Since the complement of $U'$ in $\aff_{n,n-1}$ is semialgebriac, quantifier elimination shows the complement of $U''$ in $\aff_{n,d}$ is semialgebraic as well. Again, up to a homeomorphism, $\Psi_{n,d}$ is a projection map. Since preimages of dense sets under projections are dense, $U''$ is dense. Since $U''$ is dense and $\aff_{n,d} - U''$ is semialgebraic, $\aff_{n,d} - U''$ has dimension less than $nd$, implying that moreover $\PC - U''$ is semialgebraic with dimension less than $nd$. Taking the algebraic closure of this set we get a set $A$ with dimension less than $nd$. Hence $X :=\PC -  A$ is a dense, Zariski open set. Note that $X \subseteq U''$. Given $P \in X$, we claim $P$ is circumsphere rigid. Otherwise, there exists a sequence $P_1,P_2, \ldots$ approaching $P$ of point clouds such that $(P_i,H)$ is equivalent to $(P,H)$, but $P_i$ is not congruent to $P$. By continuity of $\Psi_{n,d}$, the sequence $\{\Psi_{n,d}(P_i)\}_{i=1}^\infty$ approaches $\Psi_{n,d}(P)$. Moreover, since $\Psi_{n,d}$ is an isometry of point clouds, each $\Psi_{n,d}(P_i)$ is not congruent to $\Psi_{n,d}(P)$, while $(H,\Psi_{n,d}(P_i))$ is equivalent to $(H,\Psi_{n,d}(P))$. This contradicts that $\Psi_{n,d}(P)$ is circumsphere rigid, since $P \in U''$. Hence $V$ is a dense, Zariski open set on which every $P$ is circumsphere rigid.

Since $X$ depends on $H$, we conclude the proof by again setting $V$ to be the intersection of the sets $X$ given by each $H$. Since there are only finitely many hypergraphs with edges of between $2$ and $d+1$ endpoints on $n$ vertices, $V$ is still dense and Zariski open.
\end{proof}

This result motivates the following definition.

\begin{definition}
A hypergraph $H$ with each hyperedge consisting of hyperedges with between $2$ and $d+1$ endpoints
is \emph{generically rigid (GLR) in dimension $d$} if the circumsphere framework $(H,P)$ is rigid in dimension $d$ for generic $P$.
\end{definition}

For edges $\sigma = \{i,j\}$, the condition $\det \Delta_{\sigma(P)}\det\Lambda_{\sigma(Q)} = \det \Delta_{\sigma(Q)} \det \Lambda_{\sigma(P)}$ is the same as $-d(p_i,p_j)^2 = -d(q_i,q_j)^2$, which is equivalent to $d(p_i,p_j) = d(q_i,q_j)$. As a result, this definition of GLR extends the earlier Definition \ref{def: graph rigidity} of the GLR property to hypergraphs.

\begin{remark}
The condition on the size of the hyperedges in $H$ is necessary for the 
proof of Lemma \ref{lem: circumsphere generic} to work.  When there are larger hyperedges, the configurations 
$P$ such that the points corresponding to every hyperedge lie on a 
common sphere are a proper algebraic subset of configurations, which may have 
complicated topology or be empty.
\end{remark}
We conjecture the following.
\begin{conjecture}
\label{conj:circrigid}
Let $n$ be such that $\binom{n}{d+1} \ge dn - \binom{d+1}{2}$.  Then the complete 
$(d+1)$-uniform hypergraph is generically circumsphere rigid in dimension $d$.
\end{conjecture}

We have verified this conjecture computationally via the Jacobian test on random point clouds for $d\leq 5$.

\section{The Rigidity Theory of \v{C}ech Persistence}
\label{sec:cechrigid}
In this section we investigate how the proposed circumsphere rigidity theory applies to \v{C}ech persistence. For convenience, let $\PH(P) := \PHc(P)$ and $\Phi:= \Phi^{\C}$. Analogously to Section \ref{sec:vrrigid}, we have the following definitions.
\begin{definition}
    Let $P\in\PC$ satisfy $\PH(P) = D$. We say that $P$ is \emph{identifiable up to isometry (under \v{C}ech persistence)} if for all $Q\in \PH^{-1}(D)$, $P$ and $Q$ are isometric. 
    
    We say $P$ is \emph{locally identifiable up to isometry (under \v{C}ech persistence)} if there exists a neighborhood $U$ of $P$ in $\PC$ such that every element $Q \in U\cap\PH^{-1}(D)$ is isometric to $P$. 
\end{definition}

The point of departure from the Vietoris-Rips case is that having a certain barcode no longer only constrains pairwise distances between points, motivating the following definition. Recall that $\Phi_\sigma(P) = \rho_\sigma(P)$, the minimal enclosing radius of $\sigma(P)$.

\begin{definition}
     Let $P= (p_1,\ldots,p_n)\in \PC$ and $D = \PH(P)$. We say that $\sigma \in K(n)$ is a \emph{critical simplex} of $P$ if 
     \begin{enumerate}
        \item $\rho_\sigma(P)$ is an endpoint value of $D$ and
        \item For all strict subsimplices $\tau \subset \sigma$,  $\rho_\tau(P) < \rho_\sigma(P)$.
     \end{enumerate}
    We denote the set of critical simplices of $P$ by $H_P$ and call this the \emph{critical hypergraph of $P$}.
\end{definition}
Note that if we have an inclusion of simplices $\tau \subseteq \sigma$, then $\rho_\tau(P) \leq \rho_\sigma(P)$, so the second condition of a critical simplex really amounts to saying the equality is never attained for strict subsets of $\sigma$. 

Every singleton $\{i\}$ has $\rho_{\{i\}}(P) = 0$ for all $P$, and 0 is necessarily an endpoint of $\PH(P)$ by Lemma \ref{lem:basicH0}. In general, it may not be the case that given $\sigma \in \critS(P)$ we have $\tau \in \critS(P)$ for all $\tau \subseteq \sigma$. These two facts imply that we may view $H_P$ as a hypergraph $([n],E_P)$, which is \emph{not} guaranteed to be a simplicial complex. Note that this hypergraph does not need to be pure, i.e. the edges in $E_P$ may have varying size. However, Lemma \ref{lemma:sphereunique} and property 2 in the above definition imply that all critical simplices have at most $d+1$ elements. Therefore, $(H_P,P)$ defines a circumsphere framework.

As before, we will need that the critical hypergraph does not change within the sets $S_\preceq$. The proof is similar to that of the previous section.

\begin{lemma}
\label{lem:cechcritsame}
    If $P,Q\in S_\preceq$ then $H_P = H_{Q}$.
\end{lemma}

\begin{proof}
    Let $\sigma\in H_P$. Thus $\rho_\sigma(P)$ is an endpoint $b$ of $\PH(P)$. Since $P, Q\in S_\preceq$, there is a strictly increasing function $\psi$ such that $\rho_{\tau}(Q) = \psi \circ \rho_{\tau}(P)$ for all simplices $\tau \in K(n)$. It follows from \cite[Lemma 1.5]{leygonie2022fiber} that $\psi(b)$ is an endpoint of $\PH(Q)$. Thus $\sigma$ satisfies the first condition for being a critical simplex. For the second condition, if $\tau$ is a strict subsimplex of $\sigma$, then $\rho_\tau(P) < \rho_\sigma(P)$. Since $P,Q\in S_\preceq$, we therefore have $\rho_\tau(Q) < \rho_\sigma(Q)$. Hence $\sigma$ is a critical simplex of $Q$. Therefore $H_P \subseteq H_{Q}$. The proof of the reverse inclusion is the same with the roles of $P$ and $Q$ switched.
\end{proof}

\begin{proposition}
    \label{prop:cechcritdists}
    Let $P$ be interior to $S_\preceq$. There is a neighborhood $U$ of $P$ such that for all $Q$ in $U$, $\PH(P) = \PH(Q)$ if and only if $\rho_\sigma(P) = \rho_\sigma(Q)$ for all $\sigma \in H_P$.
\end{proposition}

The proof of this proposition is similar to that of Proposition \ref{prop:critdists}.

\begin{proof}
Set $U$ to be the interior of $S_\preceq$. It follows that $U$ contains $P$. Let $D = \PH(P)$.

($\implies$) Suppose $Q \in U$ satisfies $\PH(Q) = D$. If $\sigma\in H_P$, then $\rho_\sigma(P) = b$, for some bounded endpoint $b$ of $D$.
By Lemma \ref{lem:preorderfiber} we have $\rho_\sigma(Q) = b$.

($\impliedby$) Suppose $Q\in U$ has that for all $\sigma\in H_P$, $\rho_\sigma(P) = \rho_\sigma(Q)$. If $\tau \in K(n)$ is any simplex such that $\rho_\tau(Q) = b$, where $b$ is an endpoint of $D$, then let $\sigma$ be a minimal subset of $\tau$ such that $\rho_\sigma(Q) = \rho_\tau(Q)$. From Lemma \ref{lem:cechcritsame} we have that $\sigma \in H_{Q} = H_P$. Thus $\rho_\sigma(Q) = \rho_\sigma(P)$. Since $\rho_\sigma(Q) = \rho_\tau(Q)$ and $P,Q \in S_\preceq$ we have that $\rho_\sigma(P) = \rho_\tau(P)$. In summary

\begin{equation*}
    \rho_\tau(P) = \rho_\sigma(P) = \rho_\sigma(Q) = \rho_\tau(Q) = b.
\end{equation*}

The same argument (without needing Lemma \ref{lem:cechcritsame}) shows that if $\rho_\tau(P) = b$ where $b$ is a barcode endpoint of $\PH(P)$, then $\rho_\tau(Q) = b$. Lemma \ref{lem:preorderfiber} then implies that $\PH(P) = \PH(Q)$.
\end{proof}

Now we use the facts about circumradii and enclosing radii established in the background to write an algebraic formula for $\rho_\sigma(P)$, when $\sigma$ is a critical simplex of $P$.

\begin{proposition}
    \label{prop:circumradform}
    Let $P \in \PC$ and $\sigma\in H_P$. Then
    \begin{equation*}
        \rho_\sigma(P)^2 = -\frac{\det\Lambda_{\sigma(P)}}{2\det\Delta_{\sigma(P)}}.
    \end{equation*}
\end{proposition}

\begin{proof}
    Suppose $\sigma(P)$ is not affinely independent. Then let $k$ be the dimension of the affine span of $\sigma(P)$. It follows that $\sigma$ has more than $k + 1$ elements. Lemma \ref{lemma:sphereunique} implies that $\sigma$ has a subset of $k+1$ elements with the same enclosing radius. Hence $\sigma$ is not an element of $H_P$, a contradiction.

    Given any $\sigma\in  H_P$, Lemma \ref{lemma:sphereunique} also implies that there is a subset $\tau \subseteq \sigma$ of at most $d+1$ points such that the minimal enclosing sphere of $\sigma(P)$ is the minimal enclosing sphere \emph{and} the minimal circumsphere of $P_\tau$. Since $\rho_\tau(P) = \rho_\sigma(P)$ we must have that $\tau = \sigma$ since $\sigma \in H_P$. By the first part $\sigma(P) = \tau(P)$ is affinely independent and hence we may compute its circumradius squared via Proposition \ref{prop:circumradeqn}, since the circumradius of $\tau(P)$ is the enclosing radius of $\tau(P)$. Hence
    \begin{equation*}
        \rho_\sigma(P)^2 = -\frac{\det\Lambda_{\sigma(P)}}{2\det\Delta_{\sigma(P)}}.
    \end{equation*}
\end{proof}

Now we arrive at our analogue of Theorem \ref{thm:vrrigid} for \v{C}ech persistence.

\cechglr*

\begin{proof}
    By the genericity assumption we may take $P$ to be interior to some $S_\preceq$.

    ($\impliedby$) If $(H_P,P)$ is rigid then take a neighborhood $\tilde{U}$ as given in the definition of rigidity of a circumsphere framework. Let $U$ be the intersection of $\tilde{U}$ with the interior of $S_\preceq$. For all $Q \in U$ such that $\PH(P) = \PH(Q)$, Proposition \ref{prop:cechcritdists} implies we have that $\rho_\sigma(P) = \rho_\sigma(Q)$ for all $\sigma \in H_P$. Squaring both sides and multiplying by $-2\det\Delta_{\sigma(P)}\det\Delta_{\sigma(Q)}$ we see by Proposition \ref{prop:circumradform} that $(H_P,P)$ and $(H_{P},Q)$ are congruent circumsphere frameworks. Hence $P$ and $Q$ are isometric.

    ($\implies$) Suppose $P$ is locally identifiable. Let $U$ be the interior of $S_\preceq$. Suppose now that
    \begin{equation}
    \label{eqn:circumrigid}
    \det\Delta_{\sigma(P)}\det\Delta_{\sigma(Q)} = \det\Delta_{\sigma(Q)}\det\Delta_{\sigma(P)}    
    \end{equation}
    holds for some $Q\in U$ and all $\sigma \in H_P$. Fixing $\sigma\in H_P$, Proposition \ref{prop:circumradform} implies that $\det\Delta_{\sigma(P)}$ is nonzero. Since $H_P = H_{Q}$ by Lemma \ref{lem:cechcritsame}, Proposition \ref{prop:circumradform} also implies that $\det\Delta_{\sigma(Q)}$ is nonzero. Dividing both sides of Equation \ref{eqn:circumrigid} by $\det \Delta_{\sigma(Q)} \det \Delta_{\sigma(P)}$ we get that $\rho_\sigma(P)^2 = \rho_\sigma(Q)^2$, by Proposition \ref{prop:circumradform} again. The numbers $\rho_\sigma(P)$ and $\rho_\sigma(Q)$ are both positive, so $\rho_\sigma(P)= \rho_\sigma(Q)$. As this holds for all $\sigma \in H_P$ and $P$ is locally identifiable up to isometry, Proposition \ref{prop:cechcritdists} implies that $P$ and $Q$ are isometric. Lemma \ref{lem:isomeansconj} implies that $P$ and $Q$ are therefore congruent, since $P$ is generic.
\end{proof}

\begin{example}
    \label{ex:threepoints}
    Let $P = (p_1,p_2,p_3)\in \mathrm{PC}_{3,d}$. If $P$ is generic then the triangle with vertices $p_1,p_2$, and $p_3$ is either acute or obtuse. If this triangle is obtuse, then by reordering points we may assume without loss of generality that the angle between the vectors $p_2-p_1$ and $p_3-p_1$ is greater than $\pi /2$. This implies that $\{1.2\}$ and $\{1,3\}$ appear before $\{2,3\}$ in the induced \v{C}ech filtration of $K(3)$. Meanwhile, the obtuse angle also implies that the enclosing radius of $P$ is $\frac{1}{2}d(p_2,p_3)$, so that $\{1,2,3\}$ appears at the same value as $\{2,3\}$ in the \v{C}ech filtration of $K(3)$. In particular, the filtered complex is contractible for after the appearance of both $\{1,2\}$ and $\{1,3\}$. Meanwhile, the inclusion of the edges $\{1,2\}$ and $\{1,3\}$ both decrease the number of path components by 1. Hence, $H_P$ consists of all vertices of $K(3)$ and the two edges $\{1,2\}$ and $\{1,3\}$. In particular $H_P$ is a graph and the rigidity of $(H_P,P)$ as a circumsphere framework is equivalent to the rigidity of $(H_P,P)$ as a classical rigidity framework. The graph $H_P$ is not rigid so $P$ is not locally identifiable up to isometry.

    If the points of $P$ form the vertices of an acute triangle, then the minimal enclosing sphere of $P$ is the circumsphere of $P$, and moreover, $\{1,2,3\}$ appears after all pairs in the filtration of $K(3)$. Therefore, the inclusion of each pair in $K(3)$ corresponds either to a reduction in the number of path components in the filtration or the creation of a cycle. Meanwhile the inclusion of $\{1,2,3\}$ kills this cycle. Hence every simplex in $H_P$ is critical. In particular, the graph $G$ of edges $\{i,j\}$ forms a (classically) globally rigid framework $(G,P)$, and hence $(H_P,P)$ must be globally rigid as well. We illustrate our conclusions in Figure \ref{fig:threepoints}.

    \begin{figure}[htpb]
    \centering
    \includegraphics[width=0.8\textwidth]{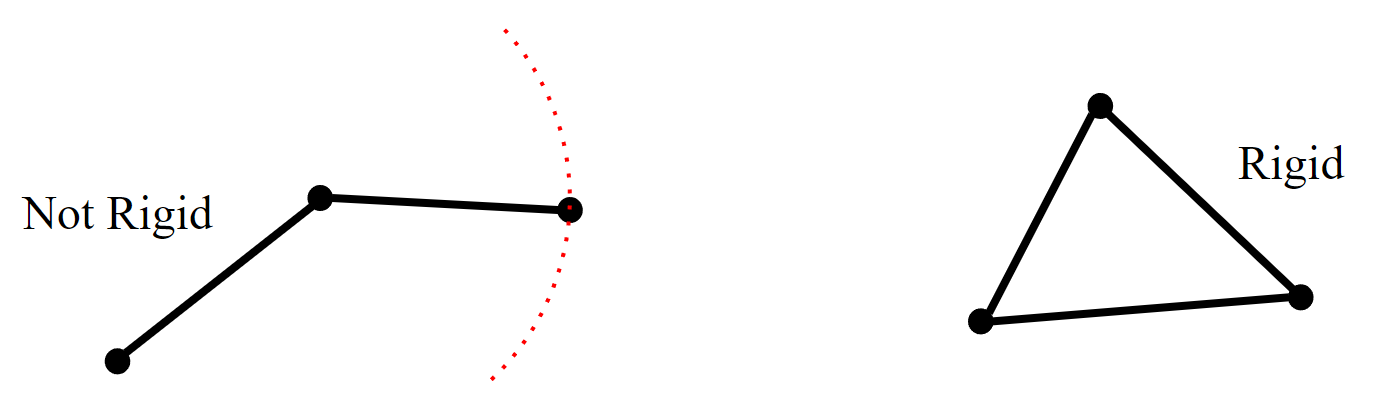}
    \caption{An illustration of Example \ref{ex:threepoints}. When the three points form the vertices of an obtuse triangle, the critical hypergraph is in fact a graph with two edges on the three points, which is flexible. When the three points form the vertices of an acute triangle, the critical hypergraph is the entire complete simplicial complex on three points, $K(3)$. In particular, the complete graph on three points is a sub-hypergraph of the critical hypergraph, implying in this case the framework $(H_P,P)$ is rigid, and in fact globally rigid.}
    \label{fig:threepoints}
\end{figure}
\end{example}

\section{Open Questions}
\label{sec:open}

In this paper we have turned questions of identifiability of the persistence map into questions in rigidity theory. Our research here leaves several natural questions in both persistence theory and rigidity theory unanswered, and we list a few of them here, some with commentary.

\begin{enumerate}
    \item For Vietoris-Rips persistence, describe sufficient conditions on a point cloud $P$ for the critical graph $G_P$ to be locally or globally rigid.
    \item For Vietoris-Rips persistence, when $n \geq d+2$ and $d\geq 2$, is our sufficient condidion for global identifiability in Theorem \ref{thm:globrig} necessary as well?

    The obstruction to proving this is as follows. It could plausibly be the case that $P = (p_1,\ldots,p_n)$ (say in the interior of some $S_\preceq$) is identifiable and yet $G_P$ is locally rigid, but not globally rigid. Local rigidity of $G_P$ would imply that that $P$ is isometric to any other point cloud in the interior of $S_\preceq$ with the same full barcode, by Proposition \ref{prop:critdists}. However non-global rigidity of $G_P$ implies there is some $Q = (q_1,\ldots,q_n) \in \PC$ that is not isometric to $P$ but has $d(q_i,q_j) = d(p_i,p_j)$ for all ${i,j}\in \crit(P)$. The problem is that $Q$ does not need to be in $S_\preceq$ and hence $Q$ may not have the same full barcode as $P$. Moreover, we have no way of proving in this situation that \emph{any} $Q \in \PC$ with $\PH(Q) = \PH(P)$ has a graph $G_{Q}$ which makes $(G_{Q}, Q)$ globally rigid. In this case we cannot apply \cite{gortler2019generic}, as we do to show our condition is sufficient. 
    \item What are combinatorial conditions on a hypergraph $H$ that makes $H$ GLR in $d$ dimensions?

    We already know, trivially, that if the edges with two vertices of $H$ form a rigid graph $G$, then $H$ must be rigid. However beyond this we have established no sufficient combinatorial condition that implies local rigidity of $H$, other than being complete with $n$ large enough 
    in dimensions at most $5$.  (See Conjecture \ref{conj:circrigid}, inter alia.)
    \item Is global rigidity in $d$ dimensions a generic property of circumsphere frameworks $(H,P)$, for $H$ fixed? If so, the previous question applies just as well to globally rigid frameworks.
    
    \item Is there a global identifiability criterion for \v{C}ech persistence analogous to Theorem \ref{thm:globrig}? In more generality, when does a filtration $\Phi$ furnish a similar global rigidity criterion? For example, what happens if we take the Vietoris-Rips filtration, but with distance between points is defined using a different norm?
\end{enumerate}

\section*{Acknowledgements}

The authors thank Michael Stillman, Jack Southgate, and Steve Oudot for conversations that helped us develop this paper and Ulrike Tillmann for early discussions leading us to study fibers of the persistence map for point clouds. DB, HAH, JL and UL received funding as members of the Center for Topological Data Analysis, Erlangen Programme for AI Hub and Center to Center collaboration between Oxford and Max Planck Institutes, funded by EPSRC grant EP/R018472/1, EP/Y028872/1 and EP/Z531224/1. DB was supported by NSF RTG-2136090. HAH gratefully acknowledges funding from a Royal Society University Research Fellowship and Renewal.

\newpage

\appendix
\section*{Appendix}
\renewcommand{\thesubsection}{\Alph{subsection}}
\subsection{Properties of Degree Zero Persistence}

Here we provide detailed proofs of the lemmas in Section \ref{sec:spanningtrees}, both restated here for convenience.

\hzerofacts*

\begin{proof}
Let $F(r)$ denote either $\C(P,r)$ or $\VR(P,r)$. Either way, the following are true
\begin{enumerate}
    \item $F(r)$ contains no elements for $r<0$
    \item For sufficiently large $r$, $F(r)=K(n)$,
    \item $F(0)$ consists only of $n$ 0-simplices, and,
    \item For $r \geq 0$, the map $H_0(F(0)) \to H_0(F(r))$ induced by inclusion of simplicial complexes is surjective.
\end{enumerate}
These facts have the following consequences for the barcode $D_0$
\begin{enumerate}
    \item No interval in $D_0$ has left endpoint less than 0,
    \item Exactly one interval in $D_0$ has right endpoint $\infty$,
    \item There are $n$ intervals in $D_0$ with closed left endpoint equal to $0$, and
    \item No interval in $D_0$ has left endpoint greater than 0.
\end{enumerate}
The filtration $F$ also has the property that for any $r$, there exists $\epsilon$ sufficiently small that $F(r) = F(r+\epsilon)$, implying that each right endpoint of $D_0$ is open. In particular, this means that the right endpoints of $D_0$ are all positive.
\end{proof}

\minspantree*

\begin{proof}
    Fix a point cloud $P = \{p_1,\ldots,p_n\}$. We have that $\PHc_0(P) = \PHvr_0(P)$ so let $D_0=\PHc_0(P)$.
    Let $G(r)$ denote the 1-skeleton of $\C(P,r)$ and note that the persistence modules $\{H_0(G(r))\}_{r\in\mathbb{R}}$ and $\{H_0(\C(P,r))\}_{r\in\mathbb{R}}$ are isomorphic, since $H_0$ and its induced maps only depend on $1$-skeleta. Define $r_i$ and $\mu_i$ as in the lemma and for our convenience we let $r_0:=0$. We construct the desired kind of minimal spanning tree on the vertex set $[n]$. First, consider the set $A_1$ of pairs $\{i,j\}$ such that $d(p_i,p_j) = 2r_1$. We fix $W_1$, a maximal cycle free subset of $A_1$, and define $E_1 := W_1$. Inductively, suppose we have constructed already $E_{k-1}$ for $k < m$. Let $A_{k}$ denote the set of pairs $\{i,j\}$ such that $d(p_i,p_j) = 2r_k$. Let $W_k$ be a maximal subset of $A_k$ subject to the constraint that $W_k \cup E_{k-1}$ is cycle free. We let $E_k := W_k \cup E_{k-1}$.

    For each $1\leq k \leq m$ we define a graph $T_k$ with vertex set $[n]$ and edge set $E_k$. We let $T_0$ denote the edgeless graph with vertices $[n]$. Since $T_k$ is a subgraph of $G(r_k)$ with the same vertex set, it has at least as many components as $G(r_k)$, and moreover each component of $T_k$ is a subset of a component of $G(r_k)$. We show by induction on $k$ that $T_k$ and $G(r_k)$ in fact have the same number of components, and therefore have the same components, with the base case $k=0$ being obvious. For $k>0$, if $T_k$ has more components than $G(r_k)$, let
    \begin{equation*}
        c : = \min \;\{d(p_i,p_j): \text{$i$ and $j$ are disconnected in $T_k$ but connected in $G(r_k)$}\}.
    \end{equation*}
    If $i$ and $j$ are disconnected in $T_k$ but connected in $G(r_k)$ then we can pick a path of edges in $G(r_k)$ connecting $i$ and $j$. At least one of these edges must connect different components of $T_k$. Taking $i'$ and $j'$ to be the incident vertices of one such edge we see that $c \leq 2r_k$. Moreover, by our inductive hypothesis, any $i$ and $j$ such that $d(p_i,p_j)\leq 2r_{k-1}$ are connected in $T_k$. Thus $ 2r_{k-1} < c\leq 2r_k$. If $ 2r_{k-1} < c < 2r_k$, pick some $i$ and $j$ in different components of $T_k$ such that $d(p_i,p_j) \leq c$. Hence $i$ and $j$ are connected in $G(c)$. From the barcode $D_0$ we deduce that $i$ and $j$ must also be connected in $G(r_{k-1})$, and therefore connected in $T_{k-1}$, a subgraph of $T_k$. This is a contradiction. Thus $c = 2r_k$, and so we may pick $i$ and $j$ disconnected in $T_k$ but connected in $G(r_k)$, such that $d(p_i,p_j) = 2r_k$. The existence of such a pair implies that $W_k$ is not maximal. Hence by contradiction we have shown that $T_k$ has the same components as $G(r_k)$.

    We let $T = T_m$. The graph $T$ is connected since from $D_0$ we deduce that $G(r_m)$ is connected. Further $T$ is cycle free by construction, so $T$ is a tree. Hence $T$ has $n-1$ edges. For each edge $\{i,j\}$ of $T$ we assign a weight $d(p_i,p_j)$. Let $w_1\leq\ldots\leq w_{n-1}$ denote these weights in ascending order. In particular, we know that each $w_i=2r_j$, for some $j$, by the construction of $T$. Moreover, $2r_k$ will occur in this sequence with the same multiplicity as the number of elements in $W_k$.

    We have already shown that $T_{k}$ has exactly as many components as $G(r_k)$ for all $k$. The structure of the barcode $D_0$ then implies that $T_{k-1}$ has $\mu_k$ more connected components than $T_k$ has. Moreover, $T_k$ is obtained from $T_{k-1}$ by attaching the edges $W_k$. Since $T_k$ is cycle free, this implies that each edge attached to $T_{k-1}$ decreases the number of components of the resulting graph by one. Hence $W_k$ has $\mu_k$ elements. This shows that $T$ is a spanning tree with multiset of edge lengths identical to the multiset of finite right endpoints in $D_0$.

    Now let $T'$ be any other spanning tree of $[n]$, and assign edge weights to $T'$ as to $T$, mapping $\{i,j\}$ to $d(p_i,p_j)$. Ordering the edge weights of $T'$, we obtain another sequence $w'_1\leq \ldots \leq w'_{n-1}$. We claim $w_i\leq w'_i$ for all $i$. Suppose otherwise and pick the smallest $l$ such that $w'_l< w_l$. Define $T(r)$ to be the subgraph of $T$ consisting of the vertex set $[n]$ and the edges of $T$ with weight less than or equal to $2r$. Note that $T(r_k) = T_k$. Similarly, define $T'(r)$ to be the subgraph of $T'$ consisting of the vertex set $[n]$ and the edges of $T'$ with weight less than or equal to $2r$. This makes both $T(r)$ and $T'(r)$ subgraphs of $G(r)$. By assumption, $T'(w'_l)$ has more edges than $T(w'_l)$. Since both are cycle free, this means that $T'(w'_l)$ has fewer components than $T(w'_l)$. Let $r_k$ be the largest of the numbers $r_0\leq\ldots\leq r_m$ that satisfies $r_k \leq w'_l$. Since $T$ only has edges with weights values $2r_i$, $T(w'_l) = T(r_k)$. Moreover, the barcode $D_0$ shows that the inclusion $G(r_k) \to G(w'_l)$ induces an isomorphism on $H_0$, so $G(r_k)$ has exactly as many components as $G(w'_l)$. By construction, $T'(w'_l)$ is a subgraph of $G(w'_k)$ with the same vertex set, so $T'(w'_l)$ has at least as many components as $G(w'_k)$. Let $C(G)$ denote the number of components of a graph $G$. The arguments in this paragraph show:
    \begin{equation*}
        C(T(w'_l)) = C(T(r_k)) = C(T_k) = C(G(r_k)) = C(G(w'_l)) \leq C(T'(w'_l)),
    \end{equation*}
    and
    \begin{equation*}
        C(T(w'_l)) > C(T'(w'_l)).
    \end{equation*}
    These statements are obviously contradictory, so $w_i \leq w'_i$ for all $i$.

    This proves that $T$ is a minimal spanning tree and that any other minimal spanning tree has the same multiset of edge lengths, as desired.
\end{proof}

\subsection{Lemmas from Real Algebraic Geometry}

Here we state two elementary, though somewhat technical, results we need from real algebraic geometry, and provide citations for the proofs.

\begin{lemma}
    \label{lem:boundarydim}
    Let $X \subseteq \mathbb{R}^m$ be semialgebraic. Then $\dim \bd (X) < m$, where $\bd$ is the topological boundary of $X$ in $\mathbb{R}^m$ with respect to the Euclidean distance. 
\end{lemma}

\begin{proof}
    This is immediate from \cite[Chapter 4, Corollary 1.10]{van1998tame}.
\end{proof}

\begin{lemma}
    \label{lem:fiberdim}
    Let $X\subseteq \mathbb{R}^n$ be semialgebraic and $f:X \to \mathbb{R}^m$ be a semialgebraic map. Then for each $d \in \{0,1,\ldots,n\}$ the sets $S_f(d):= \{a\in \mathbb{R}^m:\dim f^{-1}(a) = d\}$ and $f^{-1}(S_f(d))$ are semialgebraic and $\dim f^{-1}(S_f(d)) = \dim S_f(d) + d$.
\end{lemma}

\begin{proof}
    This follows immediately from \cite[Chapter 4, Corollary 1.6 (ii)]{van1998tame} and the fact that preimages of semialgebraic sets under semialgebraic maps are semialgebraic.
\end{proof}

\newpage

\bibliographystyle{alpha}
\bibliography{refs}

\end{document}